\documentclass{article}
\pdfoutput=1
\usepackage[utf8]{inputenc}

\usepackage{authblk}

\usepackage{amsmath,amssymb}
\usepackage{amsthm}
\usepackage{eucal}
\usepackage[svgnames]{xcolor}
\usepackage[hyphens]{url}
\usepackage{hyperref}
\usepackage[capitalise]{cleveref}
\usepackage{microtype}

\makeatletter
\renewcommand{\tableofcontents}{%
   \begin{center}
\begin{minipage}{100mm}
    
   \begin{center}
     
     \vspace{8pt}
     
       \bf{\contentsname}
   \end{center}
	
       \vspace{-18pt}
	
   \small
   \begin{center}
\@starttoc{toc}
   \end{center}	
\end{minipage}
	\end{center}
	\addvspace{2em \@plus\p@}
}

\makeatother

\providecommand{\kat}[1]{\text{\textbf{\textsl{#1}}}}

\newcommand{\isopil}{\stackrel{\raisebox{0.1ex}[0ex][0ex]{\(\sim\)}}%
			{\raisebox{-0.15ex}[0.28ex]{\(\rightarrow\)}}}

\newcommand{\RFib}{\kat{Rfib}}
\newcommand{\rfib}{\kat{Rfib}}

\newcommand{\spaces}{\mathcal{S}}
\newcommand{\QSym}{\operatorname{QSym}}

\newcommand{\name}[1]{\ulcorner #1\urcorner}

\newcommand{\DD}{D}

\newcommand{\el}{\operatorname{el}}
\newcommand{\tw}{\operatorname{tw}}
\newcommand{\sd}{\operatorname{sd}}

\newcommand{\isleftadjointto}{\dashv}

\newcommand{\PrSh}{\kat{PrSh}}
\newcommand{\Sh}{\kat{Sh}}

\newcommand{\C}{\mathcal{C}}

\newcommand{\E}{\mathcal{E}}
\newcommand{\F}{\mathcal{F}}

\newcommand{\un}{\underline}

\usepackage[utf8]{inputenc}

\usepackage{stmaryrd}

\newcommand{\inertto}{\rightarrowtail}
\newcommand{\actto}{\rightarrow\Mapsfromchar}

\DeclareRobustCommand\upperstar{%
  \mathchoice%
    {\kern0pt\raise0.55ex\hbox{$\displaystyle *$}\kern0.8pt}
    {\kern0pt\raise0.58ex\hbox{$\textstyle *$}\kern0.8pt}
    {\kern0pt\raise0.45ex\hbox{$\scriptstyle *$}\kern0.4pt}
    {\kern0pt\raise0.4ex\hbox{$\scriptscriptstyle *$}\kern0.2pt}
}%
\DeclareRobustCommand\lowerstar{%
  \mathchoice%
    {\kern0pt\raise-0.65ex\hbox{$\displaystyle *$}\kern0.8pt}
    {\kern0pt\raise-0.68ex\hbox{$\textstyle *$}\kern0.8pt}
    {\kern0pt\raise-0.55ex\hbox{$\scriptstyle *$}\kern0.4pt}
    {\kern0pt\raise-0.5ex\hbox{$\scriptscriptstyle *$}\kern0.2pt}
}%

\newcommand{\lowershriek}{_!}

\newcommand{\op}{^{\text{{\rm{op}}}}}

\newcommand{\N}{\mathbb{N}}

\newcommand{\Hom}{\operatorname{Hom}}
\newcommand{\Arr}{\operatorname{Arr}}
\newcommand{\dom}{\operatorname{dom}}
\newcommand{\codom}{\operatorname{codom}}

\newcommand{\cart}{\operatorname{cart}}

\def\overarrow#1{{\vec{#1}}}
\def\nondeg{\overarrow}

\usepackage{eucal}

\DeclareMathAlphabet{\mathbbe}{U}{bbold}{m}{n}
\newcommand{\simplexcategory}{\mathbbe{\Delta}}

\newcommand{\Deltainert}{\simplexcategory_{\operatorname{inert}}}
\newcommand{\Deltaact}{\simplexcategory_{\operatorname{act}}}

\usepackage{tikz,tikz-cd}
\usepackage{adjustbox}

\newcommand{\drpullback}{\arrow[phantom]{dr}[very near start,description]{\lrcorner}}

\tikzset{
  act /.tip = >|
}
\tikzset{
  onto/.style={/tikz/commutative diagrams/twoheadrightarrow}
}
\tikzset{
  into/.style={/tikz/commutative diagrams/rightarrowtail}
}

\newtheorem{lemma}{Lemma}[subsection]
\newtheorem{prop}[lemma]{Proposition}
\newtheorem{thm}[lemma]{Theorem}
\newtheorem{theorem}[lemma]{Theorem}

\newtheorem{cor}[lemma]{Corollary}

\theoremstyle{definition}

\newtheorem{taller}[lemma]{$\!\!$}
\newtheorem{example}[lemma]{Example}
\newenvironment{blanko}[1]%
{\begin{taller}{\normalfont\bfseries  #1}\normalfont}%
{\end{taller}}

\usetikzlibrary{decorations.markings,intersections}
\tikzset{%
scalearrow/.style n args={3}{
  decoration={
    markings,
    mark=at position (1-#1)/2*\pgfdecoratedpathlength
      with {\coordinate (#2);},
    mark=at position (1+#1)/2*\pgfdecoratedpathlength
      with {\coordinate (#3);},
    },
  postaction=decorate,
  }
}

\begin{document}

   \title{Free decomposition spaces}   

\author[$\star$]{Philip Hackney}

\author[$\star\star$]{Joachim Kock}

\footnotesize

\affil[$\star$]{University of Louisiana at Lafayette}

\affil[$\star\star$]{Universitat Aut\`onoma de Barcelona and Centre de Recerca 
Matem\`atica; \newline
currently at the University of Copenhagen}

\normalsize

\date{}

\maketitle

\vspace*{-4mm}

\begin{abstract}
  We introduce the notion of free decomposition spaces: they are simplicial 
  spaces freely generated by inert maps.
  We show that left Kan extension along the inclusion
  $j \colon \Deltainert \to \simplexcategory$
  takes
  general objects to M\"obius decomposition spaces and general maps to CULF
  maps. 
  We establish an equivalence of $\infty$-categories
  $\PrSh(\Deltainert) \simeq \kat{Decomp}_{/B\N}$. Although free decomposition
  spaces are rather simple objects, they abound in combinatorics: it seems that
  all comultiplications of deconcatenation type arise from free decomposition
  spaces. We give an extensive list of examples, including quasi-symmetric
  functions. 
\end{abstract}

\tableofcontents

\section*{Introduction}
\addcontentsline{toc}{section}{Introduction}

\subsection*{Background}

\noindent
{\bf M\"obius inversion.}
 The motivation for this work comes from the theory of M\"obius
  inversion in incidence algebras of posets (Rota~\cite{Rota:Moebius}),
  which is an important tool in enumerative combinatorics, as
  well as in application areas such as probability theory,
  algebraic topology, and renormalization, just to mention a
  few. The idea is generally to write down counting functions on
  suitable posets (or more precisely on their incidence
  coalgebras), and then express recursive relations in such a
  way that the solution is given in terms of convolution with
  the M\"obius function. In this work we are not directly 
  concerned with counting problems or M\"obius functions, but 
  rather with features of the overall framework.

  \bigskip
  
  \noindent
{\bf Decomposition spaces.} 
The starting point is the recent discovery that
objects more general than posets admit the
construction of incidence algebras and M\"obius inversion: these are
called {\em decomposition spaces} by G\'alvez, Kock, and Tonks
\cite{GKT1}, \cite{GKT2}, and they are the same thing as the {\em
$2$-Segal spaces} of Dyckerhoff and Kapranov
\cite{Dyckerhoff-Kapranov:1212.3563} who were motivated mainly by
representation theory and homological algebra. (The last ingredient in the 
equivalence
between the two notions, the fact that a certain unitality condition is 
automatic, was provided only recently by Feller et
al.~\cite{FGKPW}.) 

From the viewpoint of combinatorics, the importance of 
decomposition spaces is that many combinatorial coalgebras, bialgebras, and 
Hopf algebras are not directly the incidence coalgebra 
of any poset, but 
that it is often possible instead to realize them as incidence 
coalgebras of a decomposition space (as illustrated in \cref{section examples} below), and in this way make 
standard tools available. The price to pay is that the theory 
of decomposition spaces is more technical than that of posets.
In particular,  a little bit of background in simplicial 
homotopy theory and 
category theory is assumed as a prerequisite for this work, and their 
basic vocabulary is used freely. 
We refer to the introductions and preliminaries sections of
the papers \cite{GKT1,GKT2,GKT:ex} for background material on the use of 
simplicial methods for combinatorial coalgebras.

Briefly (cf.~\ref{decomp-def} below), a decomposition space is
a simplicial
space satisfying the exactness condition (weaker than the
Segal condition) stating that active-inert pushouts in the simplex
category $\simplexcategory$ are sent to homotopy pullbacks. The
active-inert factorization system was already an important ingredient
in the combinatorics of higher algebra, as in Lurie's
book~\cite{Lurie:HA}, where the terminology originates.

\bigskip

In the present work, we explore further the fundamental relationship between 
decomposition spaces and the active-inert factorization system, and 
single out a new class of decomposition space, the {\em free decomposition 
spaces},
being simplicial spaces freely generated by inert maps.
We show that most comultiplications in algebraic combinatorics
of deconcatenation type arise from free decomposition spaces.

\bigskip

\noindent
{\bf Active and inert maps for categories and higher categories.}
  The {\em active} maps in the simplex category $\simplexcategory$ are the 
  endpoint-preserving maps; the {\em inert} maps are the distance-preserving 
  maps (cf.~\ref{active-inert} below). 
  The two classes of maps form a factorization system $\simplexcategory = (\simplexcategory_{\text{active}},
\Deltainert)$ first studied (by Leinster and Berger~\cite{Berger:Cellular}) to
express the interplay between algebra and geometry in the notion
of category: for the nerve of a category, the active maps parametrize the
algebraic operations of composition and identity arrows, while the inert maps
express the bookkeeping that these operations
are subject to, namely source and target. The geometric nature of this
background fabric is manifest in the fact that the category $\Deltainert$ has a
natural Grothendieck topology, through which the gluing conditions are
expressed: a simplicial set $X$ is the nerve of a category if and only if
$j\upperstar (X)$, the restriction of $X$ along 
$j\colon\Deltainert\to\simplexcategory$, is a sheaf.
This is one form of
the classical nerve theorem, which characterizes the essential image of the nerve functor.
In
particular, the question whether a simplicial set is the nerve of a category
depends only on the inert part.
This viewpoint is the starting point for the Segal--Rezk approach to
$\infty$-categories, defined by replacing simplicial sets by
simplicial spaces, and considering the sheaf condition up to homotopy.
Many other developments exploit the active-inert machinery to obtain nerve 
theorems in fancier contexts, including Weber's
extensive theory of local right adjoint monads and monads with
arities, with abstract nerve theorems~\cite{Weber:TAC13},
\cite{Weber:TAC18}, \cite{Berger-Mellies-Weber:1101.3064}, as well as
special nerve theorems for specific operad-like structures (polynomial
monads in terms of trees~\cite{Kock:0807},
\cite{Gepner-Haugseng-Kock:1712.06469}, properads in terms of directed
graphs~\cite{Kock:1407.3744}, modular operads and wheeled properads as
well as infinity 
versions~\cite{Hackney-Robertson-Yau},~\cite{HACKNEY2020107206},
and so on); see \cite{Hackney:2208.13852} for a survey.
Recently Chu
and Haugseng~\cite{Chu-Haugseng:1907.03977} have even developed a
general Segal approach to operad-like structures in terms of {\em algebraic
patterns}, where the notion of active-inert factorization system is
taken as primitive.

\bigskip

\noindent
{\bf Active and inert maps for decomposition spaces.}
For decomposition spaces, the active-inert interplay is more subtle than for categories, and the
exactness condition that characterizes them can no longer be measured on the inert maps
alone. It is now about {\em decomposition} of `arrows' rather than composition.
Roughly, the active maps encode all possible ways to decompose arrows, and the
inert maps then separate out the pieces of the decomposition. The exactness
condition characterizing decomposition spaces combines active and inert maps. It
can be interpreted as a locality condition, stating roughly that the possible
decompositions of a local region are not affected by anything outside the
region~\cite{GKT:restr}, \cite{GKT:ex}.

\bigskip

\noindent
{\bf Active and inert maps for morphisms.}
Turning to morphisms, the situation is more complex for 
decomposition spaces than for categories. For categories, the nerve functor is 
fully faithful, meaning that all simplicial maps are relevant: the simplicial 
identities for simplicial maps simply say that source and target, composition 
and identities are preserved. For decomposition spaces, this is no longer the 
case, as there are different ways in which a simplicial map could be said to 
preserve decompositions. The most well-behaved class of simplicial maps in this
respect are the CULF maps~\cite{GKT1} (standing for `conservative' and `unique 
lifting of factorizations'), a class of maps well studied in category 
theory, originating with Lawvere's work on dynamical 
systems~\cite{Lawvere:statecats}, and exploited in computer science 
in the algebraic semantics of processes~\cite{Bunge-Niefield}, \cite{Johnstone:Conduche'}, 
 \cite{Bunge-Fiore}. From the viewpoint of combinatorics the interest 
 in 
CULF maps is that they preserve decompositions in a way such as to induce coalgebra homomorphisms
between incidence coalgebras \cite{Leroux:1976}, 
\cite{Lawvere-Menni}, \cite{GKT1}. 
The formal characterization of CULF maps is that
when interpreted as natural transformations, 
they are cartesian on active maps. Independently of the
coalgebra interpretation, this pullback condition interacts very well with the
exactness condition characterizing decomposition spaces.

\bigskip

\subsection*{Contributions of the present paper}

In the present work, the focus is not so much on restriction along $j \colon
\Deltainert \to \simplexcategory$ as in the Segal case, but rather on its left
adjoint $j\lowershriek \isleftadjointto j\upperstar$, left Kan extension along
$j$. Given a presheaf $A\colon\Deltainert\op\to\spaces$ (with 
values in spaces), we are thus concerned with the 
simplicial space $j\lowershriek (A)\colon \simplexcategory\op\to\spaces$.
It is rarely the case that $j\lowershriek (A)$ is Segal.
We show instead that $j\lowershriek (A)$ is always a decomposition 
space:

\bigskip

  {\bf Theorem.} (Cf.~Proposition~\ref{prop:culf} and Corollary~\ref{j=decomp}.)
  {\em For any $A\colon \Deltainert\op\to\spaces$, the left Kan extension
  $j\lowershriek (A)\colon \simplexcategory\op\to\spaces$ is a 
  M\"obius decomposition space, and for any map $A' \to A$ of 
  $\Deltainert\op$-diagrams, we have that $j\lowershriek (A')\to 
  j \lowershriek (A)$ is CULF.}
  
  \bigskip

The decomposition spaces that arise with $j\lowershriek$ we call {\em free
decomposition spaces}.

\bigskip

A key to understand the relationship between $\Deltainert$-presheaves and 
decomposition spaces is the action of $j\lowershriek$ on the terminal presheaf $1 
\in \PrSh(\Deltainert)$. The following result is just a calculation:

\bigskip

  {\bf Lemma~\ref{lemma:BN}.}
  {\em We have $j\lowershriek(1) \simeq B\N$, the classifying space of the natural numbers.}

\bigskip

It follows that free decomposition spaces admit a CULF map 
to $B\N$. In fact this feature characterizes free decomposition spaces:

\bigskip

{\bf Proposition.} (Cf.~Corollary~\ref{cor:free}.)
{\em A decomposition space is free if and only if it admits a CULF map to 
$B\N$.
}

\bigskip

We derive 
this from the following more precise result.

\bigskip

  {\bf Theorem~\ref{main thm free}.}
  {\em The $j\lowershriek$ functor induces an equivalence of 
  $\infty$-categories}
  $$
  \PrSh(\Deltainert) \simeq \kat{Decomp}_{/B\N} .
  $$
  
\noindent
Here $\kat{Decomp}$ is the $\infty$-category of decomposition spaces and CULF 
maps, and $\kat{Decomp}_{/B\N}$ its slice over $B\N$.

The proof of this result turned out to be quite involved, 
and ended up developing into a proof of the following
very general result, which is an $\infty$-version of a
theorem of Kock and Spivak~\cite{Kock-Spivak:1807.06000}:

\bigskip

{\bf Theorem~\ref{thm:untwist}.}
{\em
  For $\DD$ a decomposition space, there is a natural 
equivalence of $\infty$-categories}
$$
\kat{Decomp}_{/\DD} \simeq \rfib(\tw\DD) .
$$

\noindent 
Here $\tw(\DD)$ denotes the edgewise subdivision of the simplicial space $\DD$
--- when $\DD$ is a (Rezk complete) decomposition space, this is an $\infty$-category~\cite{HK-untwist}, called the
twisted arrow category of $\DD$. The right-hand side $\rfib(\tw\DD)$ is the
$\infty$-category of right fibrations over $\tw(\DD)$, which is equivalent to
$\PrSh(\tw\DD)$ under the basic equivalence between  
right fibrations and presheaves (see for example 
\cite[Theorem~3.4.6]{AyalaFrancis:Fibrations}). 

\bigskip

In order to apply the general theorem, take 
$\DD=B\N$,
and note the following:

\bigskip

{\bf Lemma~\ref{lemma twist BN}.}
  {\em There is a natural equivalence of categories}
  $$
  \Deltainert \simeq  \tw(B\N)  .
  $$

  Theorem~\ref{main thm free} follows essentially from this observation 
and the general theorem, but there is still some work to do to show
that in this special case, the untwisting of the general theorem can
actually be identified with left Kan extension along $j$, surprisingly.

\bigskip

Since the general theorem is of independent interest, and since the proof 
is very long, we have separated it out into a paper on its 
own~\cite{HK-untwist}.

\bigskip

Having characterized free decomposition spaces as those admitting a 
CULF map to $B\N$, it is interesting to know that such a map is 
unique if it exists. This statement is equivalent to the following 
result.

\bigskip

  {\bf Proposition~\ref{proposition ff}.}
  {\em The forgetful functor $\kat{Decomp}_{/B\N} \to \kat{Decomp}$ is fully faithful.}

\bigskip

Together with \cref{main thm free}, this implies a fundamental property of 
$j\lowershriek$:

\bigskip

  {\bf Corollary~\ref{cor j! ff}.}
  {\em The functor $j\lowershriek \colon \PrSh(\Deltainert) \to \kat{Decomp}$ is fully faithful.}

\bigskip

Theorem~\ref{main thm free} readily implies the 
following classical and more special result due to 
Street~\cite{Street:categorical-structures}: {\em a category admits a CULF functor 
to $B\N$ if and only if it is free on a directed graph}.

We also characterize a large class of free decomposition spaces in terms of
a class of species called {\em restriction $\mathbb{L}$-species} 
(Proposition~\ref{prop:restrL}).

\bigskip

Although free decomposition spaces are rather simple, they abound in 
combinatorics. Generally it seems that all comultiplications of 
deconcatenation type (i.e.~in the spirit of splitting a word into 
a prefix and a postfix) arise from free decomposition spaces. In 
Section~\ref{section examples} we 
illustrate and substantiate this principle by giving a long list of
examples of deconcatenation-type comultiplications and the free 
decomposition spaces they are incidence coalgebras of. This includes 
many variation on paths and words, including parking functions, Dyck 
paths, noncrossing partitions, as well as processes of transition systems, Petri 
nets, and rewrite systems. In particular, the comultiplication of the
Hopf algebra of quasi-symmetric functions $\QSym$ is shown to be the 
incidence coalgebra of a free decomposition space $Q$, namely that of 
words in the alphabet $\N_+$.

\section{Preliminaries}

We run through some standard material just to set up notation.

\subsection{Active and inert maps}
As usual, $\simplexcategory$ denotes the category of finite nonempty ordinals
$$
[n] = \{0 < 1 < \cdots < n\} .
$$
This category admits a presentation by generators (the coface and codegeneracy maps; see \eqref{eq active inert generators} below) and relations (the cosimplicial identities).

\begin{blanko}{The active-inert factorization system.}\label{active-inert}
  The category $\simplexcategory$ has an active-inert factorization system:
every map factors uniquely as an active map followed by an inert map.
  The {\em active maps}, written $g\colon[k]\actto [n]$, are those that  
  preserve end-points,
  $g(0)=0$ and $g(k)=n$; the {\em inert maps}, written $f\colon 
  [m]\rightarrowtail [n]$, are
  those that are distance preserving,
  $f(i{+}1)=f(i)+1$ for $0\leq i\leq m-1$.  
  The active maps are generated by
  the codegeneracy maps and the inner coface maps; the inert maps are
  generated by the outer coface maps $d^\bot$ and $d^\top$.
This is illustrated in the following diagram of generating maps of $\simplexcategory$.
\begin{equation}\label{eq active inert generators}
\begin{tikzcd}[sep=large]
  {[0]} \rar["d^1", shift left=1.5, color=blue, tail] \rar["d^0"', shift right=1.5, color=blue, tail] &
  {[1]} \rar["d^2", shift left=3, color=blue, tail] \rar["d^0"', shift right=3, color=blue, tail] 
\lar[shorten <= 0.5em, shorten >= 0.5em, color=purple, -act]
\rar[color=purple, -act]
  &
  {[2]}  \rar["d^3", shift left=4.5, color=blue, tail] \rar["d^0"', shift right=4.5, color=blue, tail] 
\lar[shift left=1.5,shorten <= 0.5em, shorten >= 0.5em, color=purple, -act]
\lar[shift right=1.5,shorten <= 0.5em, shorten >= 0.5em, color=purple, -act]
\rar[shift left=1.5, color=purple, -act] \rar[shift right=1.5, color=purple, -act]
  & 
  {[3]} 
\rar["d^4", shift left=6, color=blue, tail] \rar["d^0"', shift right=6, color=blue, tail] 
\rar[shift left=3, color=purple, -act]
\rar[shift right=3, color=purple, -act]
\rar[color=purple, -act]
\lar[shift left=3, shorten <= 0.5em, shorten >= 0.5em, color=purple, -act]
\lar[shift right=3, shorten <= 0.5em, shorten >= 0.5em, color=purple, -act]
\lar[shorten <= 0.5em, shorten >= 0.5em, color=purple, -act]
  & \cdots 
\lar[shift left=4.5, shorten <= 0.5em, shorten >= 0.5em, color=purple, -act]
\lar[shift right=4.5, shorten <= 0.5em, shorten >= 0.5em, color=purple, -act]
\lar[shift left=1.5, shorten <= 0.5em, shorten >= 0.5em, color=purple, -act]
\lar[shift right=1.5, shorten <= 0.5em, shorten >= 0.5em, color=purple, -act]
  \end{tikzcd}
\end{equation}
  (This orthogonal factorization system is an instance of the 
  important general notion of generic-free factorization system of
  Weber~\cite{Weber:TAC18}, who referred to the two classes as generic
  and free. The active-inert terminology is due to
  Lurie~\cite{Lurie:HA}.)
\end{blanko}

\begin{blanko}{Active maps vs.~$k$-tuples.}\label{Act=N^k}\label{gamma}
  For fixed  $k\in \N$, write $\operatorname{Act}(k)$ for the set of active maps out of 
  $[k]$.
  For each $n$ there is a unique active map $[1]\actto [n]$, which implies that $\operatorname{Act}(1) \cong \mathbb{N}$.
    For $i=1,\ldots,k$, write $\rho_i : [1] \rightarrowtail [k]$ for the inert map whose image is $i-1,i$.
  For an active map $\alpha: [k] \actto [n]$, write $[n_i]$ for the 
  ordinal appearing in the active-inert factorization of $\alpha 
  \circ \rho_i$:
  \[
  \begin{tikzcd}
  {[1]} \ar[d, rightarrowtail, "\rho_i"'] \ar[r, dotted, -act] & {[n_i]} \ar[d, 
  rightarrowtail, dotted, 
  "\gamma^\alpha_i"] \\
  {[k]} \ar[r, -act, "\alpha"']& {[n]}
  \end{tikzcd}
  \]
  This defines a $k$-tuple $(n_1,\ldots,n_k)\in \N^k$.
  
  Conversely, given a $k$-tuple $(n_1,\ldots,n_k) \in \N^k$, define an active 
  map  $\alpha\colon [k] \actto [n]$ (with $n:=\sum_{1\leq i\leq k} n_i$) by sending
  $j\in [k]$ to $\sum_{1\leq i\leq j} n_i \in [n]$. Clearly $\alpha(0)=0$ and 
  $\alpha(k)=n$, so $\alpha$ is indeed active.

  These assignments are inverse to each other so as to define a bijection
  $$
  \operatorname{Act}(k) \cong \N^k .
  $$
  (Below (in \ref{lem:ArractDeltacart=elBN}) we shall vary $k$ and fit these bijections into an isomorphism of 
  categories.)
\end{blanko}

\subsection{Decomposition spaces and incidence coalgebras}

\begin{blanko}{Decomposition spaces.}\label{decomp-def}
  Active and inert maps in $\simplexcategory$ admit pushouts along each other,
  and the resulting maps are again active and inert.  A {\em
  decomposition space} \cite{GKT1} is a simplicial 
  space\footnote{As is common in modern categorical homotopy theory, `space' 
  means $\infty$-groupoid. The collection of spaces forms the $\infty$-category $\spaces$, 
  and simplicial spaces and so on therefore also form $\infty$-categories.
  The amount of $\infty$-category theory needed in this paper is rather 
  limited, and we can work with $\infty$-categories 
  model independently, as is often done in the theory of decomposition 
  spaces~\cite{GKT1}, \cite{GKT2}, \cite{GKT3}. For specificity, the reader can 
  take the model of quasi-categories of Joyal~\cite{quadern45} and 
  Lurie~\cite{HTT}, where in particular, $\infty$-groupoid means Kan 
  complex.} 
  $X\colon \simplexcategory\op\to\spaces$ that
  takes all such active-inert pushouts to pullbacks:
  \[
  X\left(
  \begin{tikzcd}
  	{[n']} \drpullback &  {[n]} \ar[l, >->]\\
		{[m']} \ar[u, ->|] & 
		{[m]} \ar[u, ->|] \ar[l, >->]
  \end{tikzcd}
  \right)
  \qquad=\quad
  \begin{tikzcd}
  	X_{n'}\ar[r]\ar[d]&
		X_n\ar[d]\\
		X_{m'}\ar[r]&
		X_{m}\ar[ul, very near end, phantom, "\lrcorner"]   .
  \end{tikzcd}
  \]
  Every Segal space is also a decomposition space~\cite[Prop.~3.7]{GKT1},
  \cite[Prop.~2.5.3]{Dyckerhoff-Kapranov:1212.3563}. In particular, posets and categories 
  are decomposition spaces, via the nerve construction.
\end{blanko}

\begin{blanko}{Incidence coalgebras.}
  The motivation for the notion of decomposition space is that they admit the 
  construction of coassociative coalgebras~\cite{GKT1}, \cite{GKT2}, generalizing the classical theory of
  incidence coalgebras of posets developed by Rota~\cite{Rota:Moebius}, 
  \cite{Joni-Rota} in the 1960s. Just as incidence coalgebras of posets are 
  spanned linearly by the poset intervals, the incidence coalgebra of a 
  decomposition space $X$ is spanned linearly by $X_1$. The comultiplication
  (which generalizes the case of posets) is given by (for $f\in X_1$)
  $$
  \Delta(f) = \sum_{\underset{d_1(\sigma)=f}{\sigma\in X_2}} d_2(\sigma) 
  \otimes d_0(\sigma) 
  $$
  which verbalizes into `sum over all $2$-simplices with long edge $f$ and 
  return the two short edges.'
\end{blanko}

\begin{blanko}{CULF maps.}
  A simplicial map $F\colon Y\to X$ between simplicial spaces is 
  called \emph{CULF}~\cite{GKT1} (standing for {\em conservative} and having 
  {\em unique lifting of factorizations}) when it is cartesian on active maps 
  (i.e.~the naturality squares on active maps are pullbacks).
  If $X$ is a decomposition space (e.g.~a Segal space) and 
  $F\colon Y \to X$ is CULF,
  then also $Y$ is a decomposition space (but not in general
  Segal).
  We denote by $\kat{Decomp}$ the $\infty$-category of decomposition spaces and 
  CULF maps.
  
  From the viewpoint of incidence coalgebras, the interest in CULF maps is that
  the incidence coalgebra construction is functorial (covariantly) in CULF 
  maps~\cite{GKT1}.
\end{blanko}

\begin{blanko}{Finiteness conditions and M\"obius decomposition spaces.}\label{blanko finiteness}
  The theory of incidence coalgebras of decomposition spaces
  is natively `objective,' meaning that the constructions deal 
  directly with the combinatorial objects rather than with the vector spaces they 
  span~\cite{GKT1}. At this level the theory does not require any finiteness conditions.
  However, in order to be able to take (homotopy) cardinality to arrive at
  coalgebras in vector spaces as in classical combinatorics, it is necessary
  to impose certain finiteness conditions~\cite{GKT2}.
   A decomposition space is {\em locally finite} when the maps $X_0 \stackrel{s_0}\to X_1 
  \stackrel{d_1}\leftarrow X_2$ are (homotopy) finite~\cite{GKT2}. Often they are also 
  discrete, in which case $X$ is called {\em locally discrete}. The first condition 
  ensures that the general incidence-coalgebra construction admits a (homotopy) 
  cardinality. The discreteness condition ensures that the sum formula for 
  comultiplication is free from denominators. An important class of locally finite
  decomposition spaces, called {\em M\"obius decomposition spaces}, are 
  characterized by (a completeness condition and) one further finiteness condition, 
  namely that for every 
  $1$-simplex $f\in X_1$ there are only finitely many higher nondegenerate 
  simplices with long edge $f$ (recall that the long edge of $\sigma \in X_n$ is the pullback of $\sigma$ along the unique active map $[1] \actto [n]$). As suggested by the terminology, M\"obius decomposition spaces admit a 
  M\"obius inversion principle~\cite{GKT2}. M\"obius decomposition spaces that 
  are nerves of ordinary categories are precisely the M\"obius categories 
  of Leroux~\cite{Leroux:1976}. When the category is a poset, the 
  condition is equivalent to being locally finite, as in the classical 
  theory of Rota~\cite{Rota:Moebius}.
  
  We shall not need to verify any of these finiteness conditions directly. We shall only 
  need the fact \cite{GKT2} that if $Y \to X$ is CULF and if $X$ is locally 
  finite, locally discrete, or M\"obius, then so is $Y$. (And then we shall 
  invoke the fact that $B\N$ has all three properties.)
\end{blanko}

\section{Free decomposition spaces}

\subsection{Left Kan extension along \texorpdfstring{$j$}{j}}\label{ss lke}

The input data for free construction is a $\Deltainert$-presheaf; the category $\Deltainert$ admits a presentation as follows (compare to \eqref{eq active inert generators} on page \pageref{eq active inert generators}).
\begin{equation}\label{eq deltainert}
  \begin{tikzcd}
  {[0]} \ar[r, "d^\top", shift left] \ar[r, "d^\bot"', shift right] &
  {[1]} \ar[r, "d^\top", shift left] \ar[r, "d^\bot"', shift right] &
  {[2]}  \ar[r, "d^\top", shift left] \ar[r, "d^\bot"', shift right] & 
  {[3]}  \ar[r, "d^\top", shift left] \ar[r, "d^\bot"', shift right] & \cdots & d^\top d^\bot = d^\bot d^\top .
  \end{tikzcd}
\end{equation}
Throughout, $j\colon \Deltainert \to \simplexcategory$ denotes the inclusion 
functor.
Our immediate goal is to give a formula for the left Kan extension along $j$ in \cref{j_!A}.
We first prove a general lemma concerning factorization systems.

\begin{blanko}{Factorization systems and cartesian fibrations.}
  Let $\C$ be an $\infty$-category with a factorization system
  $(\E,\F)$. Let $\Arr(\C) = \operatorname{Fun}(\Delta^1,\C)$ be the 
  $\infty$-category of arrows, with the cartesian fibration
  $\dom  \colon \Arr(\C) \to \C$. Consider now the full subcategory 
  spanned by the arrows in $\E$, denoted $\Arr^\E(C)$.  This is again
a cartesian fibration, and the cartesian arrows are
those for which the codomain component is in $\F$ \cite[Lemma 1.3]{GKT3}. 
If we take only those,
we thus get a right fibration $\Arr^\E(\C)^{\cart} \to \C$. 
This corresponds to a presheaf $\C\op\to\spaces$, sending  an object $x$
to the space of $\E$-arrows out of $x$. 
(The inclusion $\Arr^\E(\C) \to \Arr(\C)$ does not preserve cartesian 
arrows, but it has a right adjoint which sends an arrows to its 
$\E$-factor, and this right adjoint does preserve cartesian arrows.)
\end{blanko}

\begin{lemma}\label{lem fact system act}
Suppose that $(\E,\F)$ is a factorization system on  $\C$.
  For a presheaf $A\colon \F\op \to \spaces$ with corresponding right fibration
  $\widetilde A \to \F$, the left Kan extension along $j\colon \F \to \C$ corresponds to the 
  left-hand composite in the pullback diagram
    $$\begin{tikzcd}
j\lowershriek (\widetilde{A}) \drpullback \ar[r] \ar[d] &  \widetilde{A} \ar[d] \\
\Arr^{\E}(\C)^{\cart} 
\ar[d, "\dom"'] \ar[r, "\codom"'] & 
\F \\
\C &
\end{tikzcd}
$$

\end{lemma}
\begin{proof}
The codomain functor $\codom$ admits a right adjoint $s$ 
(which is also a right inverse): $s$ sends an object $x$ to the 
identity arrow on $x$.
Since $\codom$ is a left adjoint of $s$, we have $\codom\!\upperstar 
= s\lowershriek$ as functors $\RFib(\F) \to 
\RFib(\Arr^{\E}(\C)^{\cart})$.
Further, $\dom \circ s = j$, hence $j\lowershriek = \dom\lowershriek 
\circ s\lowershriek = \dom\lowershriek \circ \codom\!\upperstar$. 
\end{proof}

Instantiating at the active-inert factorization system from \ref{active-inert}, we obtain the following.

\begin{cor}\label{j! as pbk}
    For a presheaf $A\colon \Deltainert\op \to \spaces$ with corresponding right fibration
  $\widetilde A \to \Deltainert$, the left Kan extension along $j \colon 
  \Deltainert \to \simplexcategory$ corresponds to the 
  left-hand composite in the pullback diagram
  $$\begin{tikzcd}
j\lowershriek(\widetilde A) \drpullback \ar[r] \ar[d] &  \widetilde A \ar[d] \\
\Arr^{\operatorname{act}}(\simplexcategory)^{\cart} 
\ar[d, "\dom"'] \ar[r, "\codom"'] & 
\Deltainert \\
\simplexcategory
\end{tikzcd}
$$
\end{cor}
Explicitly, $\Arr^{\operatorname{act}}(\simplexcategory)^{\cart} $ is the 
category whose objects are the active maps in $\simplexcategory$ and 
whose
arrows from $\alpha'$ to $\alpha$ are commutative squares
\[
\begin{tikzcd}
{[k']} \ar[d, -act, "\alpha'"'] \ar[r] & {[k]} \ar[d, -act, "\alpha"]  \\
{\phantom{,}[n']\phantom{,}} \ar[r, rightarrowtail] & \phantom{,}{[n]}.
\end{tikzcd}
\]

The corollary gives the following explicit sum-over-active-maps 
formula for $j\lowershriek 
(A)$:

\begin{cor}\label{j_!A}
  The simplicial space $X = j\lowershriek(A)$ has
$$
X_k = \sum_{[k]\actto [n]} A_n .
$$
\end{cor}
\begin{proof}
  In the diagram in \ref{j! as pbk}, the fiber over $[k]\in \simplexcategory$ is computed by first 
  expanding the fiber in $\Arr^{\operatorname{act}}(\simplexcategory)^{\cart} $
  over $[k]$, which is the set $\{[k] \actto [n]\}$ (for varying $n$). For each element 
  in this set, the fiber is clearly $A_n$.
\end{proof}

\begin{example}
Let $A$ be the $\Deltainert$-presheaf 
$ \begin{tikzcd}[cramped,sep=small]1 & 0 \ar[l, shift
  left] \ar[l, shift right]& 0\cdots \ar[l, shift
  left] \ar[l, shift right] \end{tikzcd}$.
Then $j\lowershriek(A) = 1$, the terminal simplicial space.
\end{example}

\subsection{Two identifications}

\begin{blanko}{Twisted arrow categories.}
  Recall that
  for $\C$ a small category (we shall only need the construction for 
  $1$-categories), the {\em twisted arrow category}
  $\tw(\C)$ is the category of 
  elements of the Hom functor $\C\op\times \C \to \kat{Set}$.
  It thus 
  has the arrows of $\C$ as objects, and trapezoidal commutative diagrams
\[
\begin{tikzcd}[row sep={13pt,between origins}, column sep={32pt,between origins}]
	&\cdot\\
	\cdot \ar[ur] \\
	& \\
	\cdot \ar[uu, "f'"]\\
	&\cdot \ar[ul] \ar[uuuu, "f"']
\end{tikzcd}
\]
as morphisms from $f'$ to $f$. (Fancier viewpoints in the 
$\infty$-setting play a key role in \cite{HK-untwist}, but in 
this paper only the naive viewpoint is needed.)
\end{blanko}

\begin{lemma}\label{lemma twist BN}
  There is a natural equivalence of categories
  $$
  \Deltainert \simeq \tw(B\N) .
  $$
\end{lemma}
\begin{proof}
  The category $B\N$ has only one object, and its set of arrows is $\N$.
  Therefore, the object set of $\tw(B\N)$ is $\N$, just as for $\Deltainert$.
A general map in $\tw(B\N)$ is of the form
\[
\begin{tikzcd}[row sep={13pt,between origins}, column sep={32pt,between origins}]
	&\cdot\\
	\cdot \ar[ur, "b"] \\
	& \\
	\cdot \ar[uu, "m"]\\
	&\cdot \ar[ul, "a"] \ar[uuuu, "a+m+b"']
\end{tikzcd}
\]
and it corresponds precisely to
$$
(d^\top)^{b} \circ (d^\bot)^a \  \colon [m] \longrightarrow [a{+}m{+}b]   
$$
in $\Deltainert$.
\end{proof}

This result should be well known, but we are not aware of a reference
for it. It should also be mentioned that it is also closely related to a
recent fancier result (Hoang~\cite{Hoang:2005.01198} and
Burkin~\cite{Burkin:Twisted}) stating that $\simplexcategory$ itself is the
twisted arrow category (in a certain generalized sense)
of the operad for unital associative algebras.

In the next lemma, and later in the paper, we make reference to the category of elements of a simplicial set or simplicial space $X$.
Recall that $X \colon \simplexcategory\op \to \spaces$ has an associated right fibration over $\simplexcategory$, and the category of elements is the domain of this right fibration $\el(X) \to \simplexcategory$.
The objects of $\el(X)$ are simplices $\Delta^n \to X$.

\begin{lemma}\label{lem:ArractDeltacart=elBN}
  There is a canonical identification 
  \[
\begin{tikzcd}[column sep={3.2em,between origins}]
\Arr^{\operatorname{act}}(\simplexcategory)^{\cart}
\ar[rd, "\dom"']& \quad\simeq & \el (B \N) \ar[ld]  \\
 & \simplexcategory &
\end{tikzcd}
\]
of right fibrations over $\simplexcategory$.
\end{lemma}
\begin{proof}
  The objects of $\Arr^{\operatorname{act}}(\simplexcategory)^{\cart}$
  in the fiber over $[k]\in \simplexcategory$
  are the active maps $[k]\actto [n]$, whereas the objects in $\el (B \N)$
  over $[k]$
  are $k$-tuples $(n_1,\ldots,n_k)$ of natural numbers. The bijection between
  these sets was already described in \ref{Act=N^k}.
  
  Functoriality amounts to matching up cartesian lifts
  for the two fibrations. For active maps in $\simplexcategory$, the lifts in 
  $\Arr^{\operatorname{act}}(\simplexcategory)^{\cart}$ are given by composition, and
  the lifts in $\el(B\N)$ are given by 
  addition of natural numbers. For inert maps in $\simplexcategory$, the lifts in
  $\Arr^{\operatorname{act}}(\simplexcategory)^{\cart}$ are given by 
  active-inert factorization, and the cartesian lifts in $\el(B\N)$ are given by 
  projections. The checks are routine.
\end{proof}

\subsection{\texorpdfstring{$j\lowershriek$}{j!} gives decomposition spaces 
and CULF maps}

\begin{lemma}\label{lemma:BN}
  For the terminal $\Deltainert\op$-diagram $1$ we have
  $$
  j\lowershriek(1) = B\N
  $$
  (nerve of the one-object category $\N$).
\end{lemma}

\begin{proof}
Under the 
basic equivalence
$\PrSh(\simplexcategory) \simeq \rfib(\simplexcategory)$ between presheaves and 
right fibrations, 
Lemma~\ref{lem fact system act} 
  gives us that $j\lowershriek(1)$ is 
  $\Arr^{\operatorname{act}}(\simplexcategory)^{\cart} \to 
  \simplexcategory$, and by Lemma~\ref{lem:ArractDeltacart=elBN} this 
  is equivalent to $\el( B\N) \to \simplexcategory$, as asserted.
  The result can also be proved by a direct calculation using
  Lemma~\ref{j_!A}:
In simplicial degree $k$ we have 
  $$
  j\lowershriek(1)_k = 
  \sum_{n\in \N} \Hom_{\Deltaact}([k],[n]) \cong \N^k, 
  $$
where the isomorphism sends $f\colon [k] \actto [n]$ to $(f(i) - f(i-1))_{1 \leq i \leq k}$.
  It then remains to describe the face and degeneracy maps.
\end{proof}

\begin{prop}\label{prop:culf}
If $A \to B$ is a map of $\Deltainert\op$-diagrams, then $j\lowershriek(A) 
\to j\lowershriek(B)$ is CULF.
\end{prop}
\begin{proof}
  By standard arguments (see \cite[Lemma 4.1]{GKT1}) it is enough to 
  show that for any $k\in \N$, the naturality square 
  \[
  \begin{tikzcd}
  j\lowershriek (A)_k \ar[d] \ar[r, -act] & j\lowershriek (A)_1 \ar[d]  \\
  j\lowershriek (B)_k \ar[r, -act] & j\lowershriek (B)_1 
  \end{tikzcd}
  \]
  is a pullback.
  We establish this by showing that the fibers of the two horizontal 
  maps are equivalent (for every point $x\in j\lowershriek (A)_1$). 
  Let us compute the fiber of the top map.
  By the explicit formula in Corollary~\ref{j_!A}, this map is
  $$
  \sum_{[k]\actto[n]} A_n \longrightarrow \sum_{[1]\actto[n]} A_n 
  $$
  given on the indexing sets by precomposition with the (unique) active 
  map $[1] \actto [k]$ and on the summands by the identity map $A_n 
  \to A_n$. The fiber is thus discrete, given by the finite set 
  $\Hom_{\simplexcategory_{\operatorname{act}}}([k], [n])$ of active 
  maps from $[k]$ to $[n]$. In particular, the fiber does not depend 
  on the point $x\in A_n$, and indeed does not even depend on $A$, only on $n$. 
  It is therefore the same for $B$. 
\end{proof}

\cref{lemma:BN} states that 
$j\lowershriek (1) = B\N$. Since $B\N$ is a M\"obius decomposition space
(in fact even a M\"obius category in the sense of Leroux~\cite{Leroux:1976}),
and since anything CULF over a M\"obius decomposition space is again a 
M\"obius decomposition space~\cite[Prop.~6.5]{GKT2}, it follows from \cref{prop:culf} that:

\begin{cor}\label{j=decomp}
  For any $A\colon \Deltainert\op\to\spaces$, the left Kan extension 
  $j\lowershriek(A) \colon \simplexcategory\op\to\spaces$ is a M\"obius decomposition space, called the \emph{free decomposition space associated to $A$.}
\end{cor}

\section{Main theorem}

In this section we prove the main theorem, that
left Kan extension along the inclusion 
$j \colon \Deltainert \to \simplexcategory$
induces an equivalence 
$j\lowershriek \colon \PrSh(\Deltainert) \isopil \kat{Decomp}_{/B\N}$.
We also show that in fact the CULF map to $B\N$ 
is unique if it exists,
and as a result the forgetful functor $\kat{Decomp}_{/B\N} \to \kat{Decomp}$ is fully 
faithful. Combining these results we see that
actually
$j\lowershriek \colon \PrSh(\Deltainert) \to \kat{Decomp}$
itself is fully faithful.

\subsection{Untwisting theorem}

We briefly reproduce the main result of \cite{HK-untwist}.

\begin{blanko}{Edgewise subdivision and twisted arrow categories.}
  The edgewise subdivision $\sd(X)$ of a simplicial space $X\colon 
  \simplexcategory\op\to\spaces$ is given by precomposing with the functor
  $Q\colon \simplexcategory \to \simplexcategory$, $[n] \mapsto [2n+1]$;
  see \cite{HK-untwist}. In particular, $\sd (X)_0 = X_1$ and $\sd (X)_1 = X_3$.
  When $X$ is a decomposition space then $\sd(X)$ is in fact a Segal space 
  (and conversely~\cite{BOORS:Edgewise}); it is then denoted 
  $\tw(X)$. Furthermore, $\sd$ takes CULF maps to right fibrations. When $X$ is 
  the nerve of a category (or Segal space), then $\sd(X)$ is the nerve of the 
  twisted arrow category. 

  There is a natural transformation $\lambda\colon \el \Rightarrow \tw$ from
  the category of elements to the twisted arrow 
  category~\cite{HK-untwist}, given on objects by sending
  $\Delta^n \to X$ (an object in $\el(X)$) to the composite
  $\Delta^1 \actto \Delta^n \to X$ (an object in
  $\tw(X)$).\footnote{The natural transformation $\lambda$ goes back to
  Thomason's {\em Notebook 85}~\cite{Thomason:notebook85}, where it is described in
  the special case of the nerve of a $1$-category.}
\end{blanko}

\begin{theorem}
  [\cite{HK-untwist}]
  \label{thm:untwist}
  For $\DD$ a decomposition space, there is a natural 
equivalence of $\infty$-categories

$$
\kat{Decomp}_{/\DD} \isopil \rfib(\tw\DD) .
$$
In the forward direction, it takes a CULF map $X{\to}\DD$ to $\tw(X) {\to} 
\tw(\DD)$. In the backward direction it is given essentially (modulo some translations 
involving nerves and elements)  by pullback along $\lambda$.
\end{theorem}
In more detail, given a right fibration $X \to \tw(\DD)$, in the diagram
\[
\begin{tikzcd}
\lambda\upperstar(X) \drpullback \ar[r] \ar[d, "f"'] & X \ar[d]  \\
\el(\DD) \ar[r, "\lambda"'] \ar[d] & \tw(\DD) \\
 \simplexcategory &
\end{tikzcd}
\]
the simplicial map associated to $f$ (a map of right fibrations over $\simplexcategory$) is shown to be CULF.

We shall need the theorem only in the very special case where $\DD$ is $B\N$,
and
give an 
explicit description of $\lambda$ in this case.

\subsection{Equivalence between 
\texorpdfstring{$\Deltainert\op$}{Δ-int-op}-diagrams and decomposition spaces over \texorpdfstring{$B\N$}{BN}
}

\begin{thm}\label{main thm free}
  The $j\lowershriek$ construction induces an equivalence of $\infty$-categories
  $$
  \PrSh(\Deltainert) \simeq \kat{Decomp}_{/B\N} .
  $$
\end{thm}

\begin{cor}\label{cor:free}
  A decomposition space is free if and only if it admits a CULF map to 
$B\N$.
\end{cor}

\begin{proof}[Proof of Theorem~\ref{main thm free}]
  Theorem \ref{thm:untwist} and Lemma~\ref{lemma twist BN} give us 
  equivalences
  $$
  \kat{Decomp}_{/B\N} \isopil \rfib(\tw B\N) \isopil \PrSh(\Deltainert) .
  $$
  It only remains to see that the inverse equivalence is actually given by 
  $j\lowershriek$.
  
  To this end, for $A\in \PrSh(\Deltainert)$ and $\widetilde{A} \to \Deltainert$ its associated right fibration, we need to match up the diagrams
  \[
  \begin{tikzcd}
  \lambda\upperstar(\widetilde{A}) \drpullback \ar[r] \ar[d] & \widetilde{A} \ar[d]  \\
  \el(B\N) \ar[r, "\lambda"'] \ar[d] & \tw(B\N) \\
   \simplexcategory &
  \end{tikzcd}
  \qquad \text{ and }
  \qquad
  \begin{tikzcd}
  j\lowershriek(\widetilde A) \drpullback \ar[r] \ar[d] &  \widetilde A \ar[d] \\
  \Arr^{\operatorname{act}}(\simplexcategory)^{\cart} 
  \ar[d, "\dom"'] \ar[r, "\codom"'] & 
  \Deltainert \\
  \simplexcategory
  \end{tikzcd}
  \]
  because the first diagram computes the inverse equivalence according to
  Theorem~\ref{thm:untwist}, whereas the second diagram computes $j\lowershriek$
  according to \ref{j! as pbk}. But this is the content of 
  the following lemma.
\end{proof}

\begin{lemma}
  The identifications $\Arr^{\operatorname{act}}(\simplexcategory)^{\cart}
  \simeq \el (B \N)$ (Lemma~\ref{lem:ArractDeltacart=elBN}) and $\Deltainert
  \simeq \tw(B\N)$ (Lemma~\ref{lemma twist BN}) are compatible with the maps
  $\lambda$ and $\operatorname{codom}$:
  \[ \begin{tikzcd}
  \el (B \N) \ar[r, "\lambda"] \ar[d, "\simeq"'] & \tw (B \N) \ar[d, "\simeq"] \\
  \Arr^{\operatorname{act}}(\simplexcategory)^{\cart} \ar[r, 
  "\operatorname{codom}"']  & \Deltainert   .
  \end{tikzcd} \]
\end{lemma}

\begin{proof}
  On objects, start with an element $(n_1,\ldots,n_k)$ in $\el(B\N)$ 
  (lying over $[k]\in \simplexcategory)$). The component of $\lambda$ 
  takes this to the long edge, which is just the sum $n := 
  \sum_{i=1}^k n_i$, which is an object of $\tw(B\N)$. Under the 
  identification in Lemma~\ref{lemma twist BN}, this corresponds 
  to the object $[n] \in \Deltainert$. The other way around: via the 
  identification of Lemma~\ref{lem:ArractDeltacart=elBN}, $(n_1,\ldots,n_k) \in\el(B\N)$ 
  corresponds to $[k] \actto [n]$, whose codomain is $[n] \in 
  \Deltainert$.
  
  To see that the diagram commutes also on morphisms is a bit more 
  involved. A morphism in $\el(B\N)$ is the data of $\Delta^h 
  \stackrel{\phi}\longrightarrow
  \Delta^k \stackrel{(n_1,\ldots,n_k)}\longrightarrow B\N$. 
  To compute $\lambda$ (see \cite{HK-untwist} for details),
  we first need to write down the diagram
  \[
  \begin{tikzcd}
  \Delta^1 \ar[d, -act] \ar[r, "d^\top d^\bot"] & \Delta^3 \ar[d, 
  -act]   \\
  \Delta^h \ar[r, "\phi"'] & \Delta^k 
  \end{tikzcd}
  \quad
  \adjustbox{scale=0.7}{%
  \begin{tikzcd}[column sep={10mm,between origins}]
  0 \ar[d, mapsto]& 1 \ar[d, mapsto]& 2 \ar[d, mapsto]& 3 \ar[d, mapsto]
  \\
  0 & \phi(0) & \phi(h) & k
    \end{tikzcd}
  }
  \]
  The value of $\lambda$ is now
  the $3$-simplex $\Delta^3 \to B\N$ given by composition, namely the 
  $3$-tuple
  $$
  \left(  \sum_{i=1}^{\phi(0)} n_i, 
  \sum_{i=\phi(0)+1}^{\phi(h)} n_i,
  \sum_{i=\phi(h)+1}^k n_i \right)  =: (a,m,b) .
  $$
  Under the identification in Lemma~\ref{lemma twist BN}, this is the inert map
  $[m] \stackrel{(d^\bot)^a (d^\top)^b}\inertto [n]$.
      
  The other way around: starting again with the morphism $\Delta^h \stackrel{\phi}\longrightarrow 
  \Delta^k \stackrel{(n_1,\ldots,n_k)}\longrightarrow B\N$ in $\el(B\N)$, 
  under the identification of Lemma~\ref{lem:ArractDeltacart=elBN}
  the corresponding morphism in $\Arr^{\operatorname{act}}(\simplexcategory)^{\cart} $ is the commutative square
  \[
  \begin{tikzcd}
  {[h]} \ar[d, -act] \ar[r] & {[k]} \ar[d, -act]  \\
  {[m]} \ar[r, rightarrowtail] & {[n]}
  \end{tikzcd}
  \]
  given by active-inert 
  factorization of the composite 
  $[h]\to[k]\to[n]$, so we 
  have
  $$
  m = \sum_{i=\phi(0)+1}^{\phi(h)} n_i ,
  $$
  and the inclusion is given by writing $n=a+m+b$ where
  $a=\sum_{i=1}^{\phi(0)} n_i$ and $b= \sum_{i=\phi(h)+1}^k n_i$. 
  This is the same inert map as we found above.
\end{proof}

\subsection{Fully faithfullness}

Recall from \cite[\S 6]{GKT2} that the {\em length} of a $1$-simplex
$f\in X_1$ is by definition the biggest dimension of a nondegenerate
simplex $\sigma$ such that $f$ is the long edge of $\sigma$. 
If every
$1$-simplex has finite length, then it defines the {\em
length-filtration} on $X$, which is a function $X_1 \to \N$. 
As mentioned in \ref{blanko finiteness}, a decomposition space with length filtration is M\"obius if it is furthermore locally finite.

\begin{lemma}
  If $Y \to X$ is a CULF map of decomposition spaces, and if $X$ is
  M\"obius, then $Y$ is M\"obius again, and moreover the length
  filtration of $Y$ is induced from that of $X$ (by precomposition).
\end{lemma}

\begin{proof}
  The first statement is Proposition~6.5 of \cite{GKT2}. The more
  precise statement that in fact the length filtration of $Y$ is
  induced from that of $X$ follows from the main ingredient in the
  proof of that Prop.~6.5, namely the fact (Proposition~2.11 of
  \cite{GKT2}) that CULF maps preserve and reflect nondegenerate
  simplices. But the length filtration is defined entirely in terms of
  dimensions of nondegenerate simplices.
\end{proof}

Next note that $B\N$ admits the length filtration tautologically: the
length of a $1$-simplex $n\in (B\N)_1$ is clearly $n$. Now if we are
given a CULF map $\phi \colon X \to B\N$, then by the lemma $\phi_1
\colon X_1 \to (B\N)_1 = \N$ is the length filtration. In particular,
if we have two CULF maps $\phi$ and $\psi$, then they must agree in
simplicial degree $1$, since the notion of length is intrinsic to $X$.
But if they agree in simplicial degree $1$, then
they must in fact agree in all degrees, as seen from the coincidence 
of the two commutative
squares
\[
\begin{tikzcd}[row sep={48pt,between origins}]
    X_n \ar[r] \ar[d, "\phi_n \overset{?}= \psi_n"'] & 
\prod_{i=1}^n X_1 \ar[d, "\prod \phi_1 =\prod \psi_1"] \\
    (B\N)_n \ar[r, "\cong"] & \prod_{i=1}^n (B\N)_1 .
\end{tikzcd}
\]
In conclusion we have established the following result.

\begin{lemma}\label{lemma one CULF map}
  If a decomposition space admits a CULF map to $B\N$, then it admits
  precisely one such map.
\end{lemma}

\begin{prop}\label{proposition ff}
  The forgetful functor $\kat{Decomp}_{/B\N} \to \kat{Decomp}$ is fully
  faithful. 
\end{prop}

\begin{proof}
  The mapping space from $\phi\colon X \to B\N$ to $\phi' \colon X' 
  \to B\N$ of the slice $\infty$-category $\kat{Decomp}_{/B\N}$ is calculated by the fiber
  sequence
  \[
  \begin{tikzcd}
  \operatorname{Map}_{/B\N}( \phi, \phi') \drpullback \ar[r] \dar & \operatorname{Map}(X,X') \ar[d, 
  "\operatorname{post} \phi'"] \\
  1 \ar[r, "\name{\phi}"'] & \operatorname{Map}(X,B\N) .
  \end{tikzcd}
  \]
  The bottom horizontal map, which picks out the unique map $\phi$, is
  an equivalence by \cref{lemma one CULF map}, so by pullback also the top
  horizontal map is an equivalence, which is the statement of fully
  faithfulness.
\end{proof}

In conjunction with \cref{main thm free} we get
\begin{cor}\label{cor j! ff}
  The functor $j\lowershriek \colon \PrSh(\Deltainert) \to
  \kat{Decomp}$ is fully faithful.
\end{cor}

\section{Miscellaneous results}

\subsection{Remarks about sheaves}

We have shown that $j\lowershriek (A)$ is always a decomposition space.
In this subsection we analyze under what conditions it is actually Segal.

It is well known (see e.g.~\cite{Berger:Cellular} or \cite{Kock:0807})
that $\Deltainert$ (interpreted as the category of nonempty linear 
graphs) has a Grothendieck topology for which a family of 
arrows constitute a covering when it is jointly surjective on vertices 
and edges. The elementary graphs are $[0]$ (the vertex) and $[1]$ (the 
edge), and every linear graph is canonically covered by its elementary 
subgraphs. We get in this way a canonical equivalence
$$
\PrSh(\simplexcategory_{\operatorname{el}}) \simeq
\Sh(\Deltainert)
$$
given by left Kan extension along the (full) inclusion $\simplexcategory_{\operatorname{el}}
\subset\Deltainert$ of the category of elementary graphs.
Note that $\PrSh(\simplexcategory_{\operatorname{el}})$ is the category of graphs.

\begin{prop}\label{prop:sheaf}
    The free decomposition space $j\lowershriek(A) : \simplexcategory\op\to\spaces$ is Segal if and only if $A : 
  \Deltainert\op\to\spaces$ is a sheaf.
\end{prop}

\begin{proof}  
  The Segal map for $X:=j\lowershriek (A)$ is (for each $k$) the map
  \begin{equation}\label{Segal}
  X_k \longrightarrow X_1 \times_{X_0} \cdots \times_{X_0} X_1  .
  \end{equation}
  Lemma~\ref{j_!A} gives
  $$X_k \ = \sum_{\alpha: [k] \actto [n]} A_n 
  $$
  on the left, whereas the right-hand side unpacks to
  $$
  {\textstyle
  \big(\underset{n_1}{\sum}\, A_{n_1}\big) \times_{A_0} \cdots \times_{A_0} 
  \big(\underset{n_k}{\sum}\, A_{n_k}\big)
  } \ = \!
    \sum_{(n_1,\ldots,n_k)} \!\!\big( A_{n_1} \times_{A_0} \cdots 
  \times_{A_0} A_{n_k} \big) .
$$
  The Segal map is the specific combination of inert maps
  $$
  \sum_{\alpha: [k] \actto [n]} A_n \longrightarrow 
  \sum_{(n_1,\ldots,n_k)} \big( A_{n_1} \times_{A_0} \cdots 
  \times_{A_0} A_{n_k} \big)
  $$
  given by sending the $\alpha$-summand to the 
  $(n_1,\ldots,n_k)$-summand via the map
  $$
  (\gamma^\alpha_1,\ldots,\gamma^\alpha_k)\upperstar: 
    A_n \longrightarrow  A_{n_1} \times_{A_0} \cdots 
  \times_{A_0} A_{n_k} .
  $$
  Recall from \ref{gamma} that the maps $\gamma_i^\alpha$ are the maps
  appearing in the active-inert factorization of $\alpha 
  \circ \rho_i$:
  \[
  \begin{tikzcd}
  {[1]} \ar[d, "\rho_i"'] \ar[r, dotted, -act] & {[n_i]} \ar[d, dotted, 
  "\gamma^\alpha_i"] \\
  {[k]} \ar[r, "\alpha"']& {[n]}
  \end{tikzcd}
  \]
  The families $\{\gamma^\alpha_i\mid i\in [k]\}$ are precisely the reduced 
  coverings of $[n]$, so if 
  $A$ is a sheaf, all the maps 
  $(\gamma^\alpha_1,\ldots,\gamma^\alpha_k)\upperstar$ are 
  thus equivalences, and therefore the Segal map, which is the sum of 
  them all,
  is an equivalence, which is to say that $X$ is Segal. Conversely, if
  $X$ is Segal, all these maps $(\gamma^\alpha_1,\ldots,\gamma^\alpha_k)\upperstar$
  are equivalences, ensuring that $A$ is a sheaf.
\end{proof}

This result is only slightly more precise than the following
classical result:

\begin{cor}[Street~\cite{Street:categorical-structures}] 
  A category admits a CULF functor to $B\N$ if and only if it 
  is the free category on a directed graph.
\end{cor}

The following corollary may be surprising at first sight:
\begin{cor}
  $A : 
  \Deltainert\op\to\spaces$ is a sheaf if and only if $j\upperstar 
  j\lowershriek (A) : 
  \Deltainert\op\to\spaces$ is a sheaf.
\end{cor}
\begin{proof}
  It is well known that a simplicial space $X: 
  \simplexcategory\op\to\spaces$ is Segal if and only if $j\upperstar 
  X: \Deltainert\op\to\spaces$ is a sheaf~\cite{Berger:Cellular}. 
  The result now follows from Proposition~\ref{prop:sheaf}.
\end{proof}

\subsection{Restriction \texorpdfstring{$\mathbb L$}{L}-species}

We briefly comment on the relationship between free decomposition spaces and
certain species.

\begin{blanko}{Restriction species.}
  A {\em combinatorial species} \cite{Joyal:1981},
  \cite{Bergeron-Labelle-Leroux}, \cite{Aguiar-Mahajan} is a functor
  $F:\mathbb{B} \to \kat{Set}$, where $\mathbb{B}$ is the groupoid of
  finite sets and bijections. An {\em $F$-structure} on a finite set $S$ is by 
  definition an element in $F[S]$. Standard examples of $F$-structures are
  lists, trees, permutations, and so on. The theory of species is an 
  objective approach to exponential generating functions.
  In order to get a comultiplication of
  the set of $F$-structures, $F$ should furthermore be contravariantly
  functorial in injections: Schmitt~\cite{Schmitt:hacs} thus defined a
  {\em restriction species} to be a functor $R:
  \mathbb{I}\op\to\kat{Set}$. For $G\in R[S]$ an $R$-structure on a
  finite set $S$, the comultiplication is then given by
\begin{equation}\label{G}
\Delta(G) = \sum_{A+B=S} G| A \otimes G | B
\end{equation}
where the sum is over all splittings of the underlying set into two disjoint 
subsets,
and where $G| A$ denotes the restriction of $G$ along $A \subset S$.

Schmitt's construction was subsumed in decomposition space theory in
\cite{GKT:restr}, where the following more general notion was also introduced,
covering many interesting examples. Let $\mathbb{C}$ denote the category of
finite posets and convex monotone injections, i.e.~full sub-poset 
inclusions $f\colon C \to P$ for which intermediate points between points 
in $C$ are again in $C$. A {\em directed restriction
species} is a functor $\mathbb{C}\op\to\kat{Set}$. The resulting
comultiplication formula is very similar to \eqref{G}, except that the sum is
now over splittings where $A$ is required to be a downward closed 
sub-poset and $B$ an
upward closed sub-poset.
Schmitt's restriction species are the special 
case of directed restriction species supported on discrete posets.
\end{blanko}

\begin{blanko}{$\mathbb L$-species and restriction $\mathbb 
  L$-species.}
  There is another special case, which covers many of the examples of free
  decomposition spaces listed below, namely those directed
  restriction species that are supported on linear orders. Classically
  \cite[Ch.~5]{Bergeron-Labelle-Leroux}, {\em $\mathbb L$-species} are functors
  $\mathbb L^{\operatorname{iso}} \to \kat{Set}$, 
  where $\mathbb L^{\operatorname{iso}}$ is the groupoid of linear orders and
  monotone bijections (this groupoid is of course discrete). 
  Where species provide a combinatorial theory for 
  exponential generating functions, $\mathbb{L}$-species do the same for 
  ordinary generating functions~\cite[Ch.~5]{Bergeron-Labelle-Leroux}.
  Combining the notions of directed restriction species and 
  $\mathbb{L}$-species,
  we arrive at the
  notion of {\em restriction $\mathbb L$-species}, defined to be functors
  $$
  \mathbb{L}\op\to \kat{Set} ,
  $$
  where now
  $\mathbb L$ denotes the category of linear orders and convex monotone injections.
  The point is that the category $\mathbb{L}$ has a presentation like this:
  $$
  \begin{tikzcd}
	\un 0 \ar[r] &
	\un 1 \ar[r, "d^\top", shift left] \ar[r, "d^\bot"', shift right] &
	\un 2 \ar[r, "d^\top", shift left] \ar[r, "d^\bot"', shift right] & 
	\un 3  & \cdots && d^\top d^\bot = d^\bot d^\top .
  \end{tikzcd}
  $$
  Comparing with the presentation for $\Deltainert$  (see \eqref{eq deltainert} in \S\ref{ss lke}), we have: 

\end{blanko}

\begin{prop}\label{prop:restrL}
  A restriction $\mathbb L$-species is the same thing as a functor 
  $A: \Deltainert\op\to \kat{Set}$ for which the two face maps
  $\begin{tikzcd}[cramped]A_0 & A_1 \ar[l, shift
  left] \ar[l, shift right]\end{tikzcd}$ coincide.
\end{prop}

The upshot is that restriction $\mathbb{L}$-species can be
given as input to the free decomposition space construction $j\lowershriek$.
In particular, presheaves $A\colon \Deltainert\op\to\kat{Set}$ for which $A_0 = *$
can be regarded as restriction $\mathbb{L}$-species, 
as will be the case in many of the examples below.

\section{Examples in combinatorics}\label{section examples}

The following selection of examples serves to illustrate the nature of free 
decomposition spaces
as a common origin of comultiplications of deconcatenation type.\footnote{Since these 
comultiplications are often the very simplest aspect of the combinatorial 
structure in question, we do not pretend to have made any contribution to the various 
application areas.} 
The protopypical example (cf.~\ref{words} below) is 
that of splitting a word into a prefix and a postfix in all ways, essentially undoing 
the familiar operation of concatenation.
Note however that
in many of the examples, the deconcatenation does not arise from a 
concatenation! More precisely, in these cases the free decomposition space is 
not Segal.

All the simplicial spaces in this section are actually simplicial sets.

\subsection{Paths and words}

\begin{blanko}{Graphs and paths.}
  For any directed graph (quiver) $\begin{tikzcd}[cramped]A_0 & A_1 \ar[l, shift
  left] \ar[l, shift right]\end{tikzcd}$, put
  $$
  A_n = A_1 \times_{A_0} \cdots \times_{A_0}
  A_1, \qquad n\geq 2.
  $$
  This is the length-$n$ path space of the graph. Since $A\colon 
  \Deltainert\op\to\kat{Set}$ is a sheaf
  by construction, $X:=j_!(A)$ is a category, by Proposition~\ref{prop:sheaf}; 
  it is of course the
  free category on the graph.
  
  This example is the starting point for a wealth of variations, some more exotic
  than others. As a first variation, we can decide to truncate at some length,
  so as to allow only paths that are of length at most $r$, that is, setting
  $A_n = A_1 \times_{A_0} \cdots \times_{A_0} A_1$, for $2\leq n\leq r$ and
  $A_n=\emptyset $ for $n>r$. This is clearly still a valid
  $\Deltainert\op$-diagram, but the resulting free decomposition space is no
  longer Segal. Truncations like this were considered by Bergner et
  al.~\cite{Bergner-Osorno-Ozornova-Rovelli-Scheimbauer:1609.02853}.
  An extreme case of this is to truncate at $n=1$, keeping only the graph 
  itself $\begin{tikzcd}[cramped]A_0 & A_1 \ar[l, shift
  left] \ar[l, shift right]\end{tikzcd}$ and 
  setting $A_n = \emptyset$ for $n\geq 2$. The
  resulting free decomposition space $X := j\lowershriek (A)$ was considered
  already by Dyckerhoff and Kapranov~\cite[Ex.~3.1.1]{Dyckerhoff-Kapranov:1212.3563}.
\end{blanko}

\begin{blanko}{Shifting up and shifting down.}
  Further variations can be obtained by shifting $A$ up or down before applying 
  $j\lowershriek$. For example, we can define $A'$ by shifting up all spaces, 
  putting $A'_0=*$ and letting $A'_n := A_{n-1}$.
  When taking now $X:=j\lowershriek (A')$ we get
  $$
  X_0 = *, \qquad X_1 = * + \sum_{n\in \N} A_n
  $$
  and also
  \[
    X_2 = * + \sum_{n_1 + n_2 \geq 1} A_{n_1+n_2-1} = 
	* + \sum_{n\in \N} A_n + \sum_{n\in \N} A_n + \sum_{n_1,n_2\in \N} 
	A_{n_1+n_2+1} .
  \]

  Easier to understand is the shifting down, setting $A'_n := A_{n+1}$.
  This can be described by starting with $\begin{tikzcd}[cramped]A_1 & A_2 \ar[l, shift
  left] \ar[l, shift right]\end{tikzcd}$, and then observe that $A_3 = A_1 
  \times_{A_0} A_1 \times_{A_0} A_1 = A_2 \times_{A_1} A_2$.
  We see that $A'$ is the path space on another graph, namely 
  the graph whose vertices are the old 
  edges, and where an edge from $f$ to $g$ is a $2$-path in which $f$ is the first 
  leg and $g$ is the second leg. 
  When we now set $X:=j\lowershriek (A')$
  we get for $X_1$ the set of all paths of length $\geq 1$, and for $X_2$ 
  the set of paths with a marked step. The inner face map forgets the mark.
  The face map $d_0$ deletes all steps before the marked edge $e$ and keeps $e$ 
  and all following steps. The face map $d_2$ deletes everything after $e$,
  and keeps everything up to and including $e$. 
   (This $X :=j\lowershriek (A')$ is actually (the nerve of) a preorder on the set of 
  edges: we can say that 
  $f\leq g$ if $d_0(f)= d_1(g)$.)

  Similarly the path $\Deltainert\op$-diagram of a quiver can be shifted down 
  $d$ places. Then $X_1$ is the set of all paths of length $\geq d$, and 
  $X_2$ is the set of all paths of length $\geq d$ with a marked $d$-path 
  somewhere in the middle.
  The resulting comultiplication will then comultiply a path by selecting a 
  $d$-path $\gamma$ somewhere  inside it, and then returning the prefix (including 
  $\gamma$)
  on the left, and returning the postfix (also including $\gamma$) on the right.
  The group-like elements are thus the $d$-paths.
\end{blanko}

\begin{blanko}{Words.}\label{words}
  Let $S$ be an alphabet, whose elements we call {\em symbols}, reserving the 
  term {\em letter} for the entries of a word. (Thus $cbabb$ is a $5$-letter 
  word comprising three symbols.) Although words in an alphabet can be seen as special
  cases of paths in a graph, namely in graphs with only one vertex 
  $\begin{tikzcd}[cramped]* & S \ar[l, shift
  left] \ar[l, shift right]\end{tikzcd}$, this case deserves special 
  mention for its importance in combinatorics.
  We define $W(S) \colon \Deltainert\op\to\kat{Set}$ by letting $W(S)_n$ be
  the set of words of length $n$; the face maps delete the 
  first or last letter.  In the resulting free 
  decomposition space $X:=j\lowershriek (WS)$, we have $X_1$ the set of all words, 
  and $X_2$ the set of all 
  words with a splitting, say $s_1 s_2 s_3 | s_4 s_5$, and the outer face maps
  return those prefix and postfix words. The comultiplication of the incidence 
  coalgebra is the standard deconcatenation, splitting a word in all possible 
  ways into two parts.
  
  This example is actually Segal.  But we can disturb it by the same
  techniques: truncate (say, to allow only words of length between $4$ and 
  $13$); that is, put the words of length $4$ in degree $0$, etc.
  The comultiplication in the incidence coalgebra will now sum over 
  all ways to select a (convex) $4$-letter subword (so that's a middle
  part), and then include that middle subword in both tensor factors
  of the comultiplication.
  
  An interesting special case of this is {\em nonempty words}.  This is to
  put the length-$n$ words in degree $n-1$. 
  Here the resulting comultiplication splits
  a word at a letter (not between two letters), and that letter then
  forms part of both sides of the split. (This coalgebra came up recently in 
  connection with the Baez--Dolan construction~\cite{Kock:1912.11320}.)
  This decomposition space can easily be interpreted as a category:
  the objects are the symbols, and a word from $x$ to $y$ is any
  word whose first letter is $x$ and whose last letter is $y$.
  Such arrows are composed by gluing words along their one-letter
  overlap.
\end{blanko}

\begin{blanko}{Quasi-symmetric functions, I.}\label{QSym:I}
  When the alphabet is $S=\N_+$, the set of positive integers, then the decomposition
  space 
  $$
  Q := j\lowershriek (W\N_+)
  $$ 
  is the decomposition space whose
  incidence coalgebra is the coalgebra of quasi-symmetric functions
  $\QSym$ in the $M$-basis~\cite{GKT:QSym}. Recall (for example from 
  \cite{Stanley:volII}) that quasi-symmetric functions 
  are certain power series of 
  bounded degree in countably many variables $x_1,x_2,\ldots$:  
  for each word $w \in Q_1$ of length $k$ there is a quasi-symmetric function
  $$
  M_w := \sum_{i_1 < i_2 <\cdots < i_k} x_{i_1}^{w_1} x_{i_2}^{w_2} \cdots
  x_{i_k}^{w_k} ,
  $$
  and the ring of quasi-symmetric functions is the linear span of these 
  so-called {\em monomial} quasi-symmetric functions $M_w$. The set $Q_1$ of all
  words in $\N_+$ thus indexes this {\em monomial basis}.
  For example,
  to a word such as 
  $(23114)$ corresponds the quasi-symmetric function $M_{23114} = 
  \sum_{i_1 < i_2 <\cdots < i_5} x_{i_1}^2 x_{i_2}^3 x_{i_3}^1 x_{i_4}^1 
  x_{i_5}^4$.
  Although the actual power series are important in many applications, the
  algebraic structure is governed by the combinatorics of words in $\N_+$
  (in the $M$-basis and in several related bases), and in the following we shall not 
  need the actual power series. In particular the coalgebra 
  structure of $\QSym$ can be described purely in terms of words: the 
  comultiplication is
  simply word splitting (deconcatenation).
  (This is not the most interesting aspect of the Hopf algebra of 
  quasi-symmetric functions, but it is the aspect that arises from a
  free decomposition space.)
\end{blanko}

\begin{blanko}{WQSym and FQSym.}
  Word quasi-symmetric functions and free quasi-symmetric 
  functions are certain variations of the previous example. 
  We shall not give the definitions here, but 
  content ourselves to describe the combinatorics of the comultiplications
  in these Hopf algebras.
  (The decomposition spaces here are still from \cite{GKT:QSym},
  but the $j_!$ interpretation is new.) We continue with the ordered 
  alphabet $S= \N_+$. A word is {\em packed} if whenever a
  symbol occurs in it, then all smaller symbols occur too. If the first or last
  letter of a packed word is omitted, and if that letter was the only occurrence
  of a symbol, then the result is not again a packed word, but there is a
  canonical way to `pack' it, by shifting down all bigger symbols. With this
  prescription the sets $A_n$ of packed words of length $n$
  assemble into a $\Deltainert\op$-diagram, and we
  get a free decomposition space $X:=j\lowershriek (A)$, where $X_1$ is the set of
  all packed words, and $X_2$ is the set of all packed words with a word
  splitting. The resulting comultiplication is that of the $F$-basis in the Hopf
  algebra of word quasi-symmetric functions $\operatorname{WQSym}$ 
  (cf.~\cite{Bergeron-Zabrocki}, \cite{Hivert-Novelli-Thibon:0605262}).
    
  The same arguments apply to packed words without repetition of symbols.
  A packed word without repetition is the same thing as a permutation.
  We obtain the comultiplication of the Hopf algebra $\operatorname{FQSym}$ 
  (free quasi-symmetric functions, also called the Malvenuto--Reutenauer Hopf 
  algebra, after \cite{Malvenuto-Reutenauer}), in the $F$-basis.
  
  There is a different way to assemble packed words into a free decomposition 
  space: let $A_n$ denote the set of packed words on $n$ symbols (any length), and let 
  $d_\top$ delete all occurrences of the largest symbol, and let $d_\bot$ 
  delete all occurrences of the smallest symbol (decrementing all larger 
  symbols, for the word to remain packed). The resulting free decomposition 
  space $X:= j \lowershriek (A)$ has $X_1$ the set of all packed words, and $X_2$
  the set of all packed words with a linear splitting of the set of employed 
  symbols into what we can call lower and upper symbols. Now $d_2 : X_2 \to X_1$ 
  deletes from a word all upper symbols, while $d_0$ deletes all lower symbols
  (decrementing the remaining symbols until the word is packed again).
\end{blanko}

\begin{blanko}{Parking functions.}
  A {\em parking function} (see~\cite{Stanley:PFNP}) is a word $w$ in the
  alphabet $\N_+$ such that if reordered to form a monotone word $m$, then we
  have $m_i \leq i$, $\forall i$. Let $A_n$ denote the set of parking functions
  of length $n$. The face maps are defined by deleting the first or last letter.
  This may violate the parking condition, but there is a canonical way to
  `parkify,' by shifting down all bigger symbols (see
  \cite{Novelli-Thibon:PQSym} for details). With this prescription the sets
  $A_n$ assemble into a $\Deltainert\op$-diagram, and we get a free
  decomposition space $X:=j\lowershriek (A)$. Here $X_1$ is the set of all
  parking functions, and $X_2$ is the set of all parking function with a word
  splitting. The resulting comultiplication is that of the $F$-basis in the Hopf
  algebra of parking quasi-symmetric functions $\operatorname{PQSym}$ of
  \cite{Novelli-Thibon:PQSym}.

  There is another way to assemble parking functions into a
  $\Deltainert\op$-diagram. A {\em breakpoint} of a length-$\ell$ parking
  function is an $i \in \{0,\ldots,\ell\}$ such that there are exactly $b$
  occurrences of symbols smaller or equal to $i$. For example the breakpoints of
  the parking function $162436166$ are $0$, $5$, $9$. Let $A_n$ denote the set
  of parking functions (of any length) with precisely $n+1$ breakpoints. There
  are face maps $\begin{tikzcd}[cramped]A_{n-1} & A_n \ar[l, shift left] \ar[l,
  shift right]\end{tikzcd}$ given as follows: if the breakpoints of $w$ are
  $(0=b_0,\ldots,b_n=\ell)$, then $d_\top$ deletes all occurrences of symbols
  $>b_{n-1}$, and $d_\bot$ deletes all occurrences of symbols $\leq b_1$;
  this involves parkification. The resulting $X:=j\lowershriek (A)$ has $X_1$
  the set of all parking functions, and $X_2$ the set of all parking functions
  with a chosen breakpoint. The resulting comultiplication is that of the
  $G$-basis in the Hopf algebra of parking quasi-symmetric functions
  $\operatorname{PQSym}$ (see~\cite{Novelli-Thibon:PQSym} again).
\end{blanko}

\begin{blanko}{Functorialities.}
  All the constructions are functorial: given a homomorphism of 
  graphs $G \to H$, there is induced a CULF map $X_G \to X_H$, and this works 
  for all the variations mentioned.
  
  Similarly, given a map of alphabets $S \to T$, there is induced a CULF functor
  for the associated free decomposition spaces $j\lowershriek(WS) \to
  j\lowershriek(WT)$. In particular, it is an interesting case when the alphabet
  $S$ is positively graded, meaning that it has a map to $\N_+$. This gives a
  CULF map $j\lowershriek (WS) \to j\lowershriek (W\N_+) = Q$, the decomposition space for $\QSym$,
  which we shall come back to below.
  
  For a graph $\begin{tikzcd}[cramped]V & E \ar[l, shift
  left] \ar[l, shift right]\end{tikzcd}$, consider words on the set of edges 
  $E$. Then there is a morphism of $\Deltainert\op$-diagrams from paths to 
  words. This can be seen as coming from the graph homomorphism from 
  $\begin{tikzcd}[cramped]V & E \ar[l, shift
  left] \ar[l, shift right]\end{tikzcd}$ to $\begin{tikzcd}[cramped]1 & E \ar[l, shift
  left] \ar[l, shift right]\end{tikzcd}$ (collapsing all vertices to a single 
  vertex).
\end{blanko}

\subsection{Further examples of deconcatenations}

\tikzset{
  ncone/.pic={
	\draw (0,0)--(0,0.2);
  }
}

\tikzset{
  nctwo/.pic={
    \draw (0,0)--(0,0.2);
	\draw (0.1,0)--(0.1,0.2);
  }
}

\tikzset{
  nctwoW/.pic={
    \draw (0,0.2)--(0,0)--(0.1,0)--(0.1,0.2);
  }
}

\tikzset{
  nctwoWW/.pic={
    \draw (0,0.2)--(0,0)--(0.2,0)--(0.2,0.2);
  }
}

\tikzset{
  nconeoneinsidetwoWW/.pic={
	\path (0,0) pic {ncone};
    \path (0.1,0) pic {nctwoWW}; 
	\path (0.2,0.1) pic {ncone};
  }
}

\newcommand{\nconeoneinsidetwoWW}{\,\tikz[x=1.3cm,y=1.3cm]{
  \path (0,0) pic {nconeoneinsidetwoWW};
}\,}

\begin{blanko}{Noncrossing partitions.}
  A partition $\pi=\{\pi_1,\ldots,\pi_k\}$ of a linearly ordered set 
  $\mathbf{n}=\{1,2,\ldots,n\}$ is called {\em noncrossing} where there are no
  $a,b\in \pi_i$ and $c,d\in \pi_j$ with $i\neq j$ such that $a<c<b<d$.
  For example, the partition $\{\{1\},\{2,4\},\{3\}\} =$
  \nconeoneinsidetwoWW\ is noncrossing,
    whereas  $\{\{1,3\},\{2,4\}\} =$
  \begin{tikzpicture}[scale=1.2]
  \draw (0,0.2)--(0,0)--(0.2,0)--(0.2,0.2); \draw
  (0.1,0.14)--(0.1,-0.06)--(0.3,-0.06)--(0.3,0.14); \end{tikzpicture} is
  crossing. The $a<c<b<d$ condition is just the formalization of the intuitive
  idea of crossing in such pictures (see \cite{Stanley:PFNP} for background). 
  
  Let $A_n$ denote the set of noncrossing partitions of $\mathbf{n}$.
  Given a noncrossing partition $\pi\in A_n$, one can obtain a new noncrossing 
  partition in $A_{n-1}$ by deleting the element $n\in \mathbf n$, or by deleting
  the element $1\in \mathbf{n}$ (and then standardizing by shifting all the 
  elements one down).
  These assignments define the face maps of a $\Deltainert\op$-diagram, and 
  we get thus a free decomposition space 
  $X:=j\lowershriek (A)$, where $X_1$ is the set of all noncrossing partitions,
  and $X_2$ is the set of all noncrossing partitions with a marked gap: 
  a {\em gap}  of $\mathbf{n}$ just a natural number $g$ from $0$ to $n$ 
  thought of as splitting the elements into those $\leq g$ and those $>g$.
  The resulting comultiplication is given  (for $\pi\in A_n$) by
  $$
  \Delta(\pi) = \sum_{\mathbf{a}+\mathbf{b}=\mathbf{n}} \pi_{\mid \mathbf a} 
  \otimes \pi_{\mid \mathbf b} .
  $$
  Here the sum is over the splitting of $\mathbf n$ into an initial segment and
  a final segment, and $\pi_{\mid \mathbf a}$ and $ \pi_{\mid \mathbf b}$ denote
  the restriction of $\pi$ to these two subsets (and standardizing).
\end{blanko}

\begin{blanko}{Dyck paths.}
  A {\em Dyck path} is a lattice path in $\N\times \N$ starting at $(0,0)$ and
  ending at $(2\ell,0)$ (for some $\ell\in \N$), taking only steps of type
  $(1,1)$ and $(1,-1)$. The height of a Dyck path is the maximal second
  coordinate. Let $A_n$ be the set of Dyck paths (varying $\ell$) of height $n$.
  The face maps $\begin{tikzcd}[cramped]A_n & A_{n+1} \ar[l, shift left] \ar[l,
  shift right]\end{tikzcd}$ are given by clipping the path, either at the top or
  at the bottom, and sliding the disconnected pieces left (and down) until they
  meet up again, as exemplified here at $
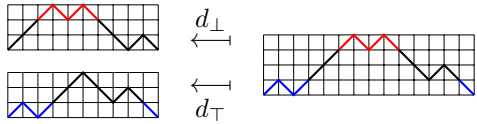
\begin{tikzcd}[cramped]A_3 & A_4 \ar[l, shift left] \ar[l,
  shift right]\end{tikzcd}$:
\begin{center}
  \begin{tikzpicture}
	\begin{scope}
	  \draw[step=0.2,black,ultra thin] (0.0,0.0) grid (2.8,0.8);
	  \draw[blue, thick] (0.0,0.0) -- (0.2,0.2) -- (0.4,0.0) -- (0.4,0.0) -- (0.6,0.2); 
	  \draw[thick] (0.6,0.2) -- (1.0,0.6);
	  \draw[red, thick] (1.0,0.6) -- (1.2,0.8) -- (1.4,0.6) -- (1.6,0.8) -- (1.8,0.6);
	  \draw[thick] (1.8,0.6) -- (2.2,0.2) -- (2.4,0.4) -- (2.6,0.2);
	  \draw[blue, thick] (2.6,0.2) -- (2.8,0.0);
	\end{scope}
	
	\begin{scope}[shift={(-3.4,0.6)}]
	  \draw[step=0.2,black,ultra thin] (0.0,0.0) grid (2.0,0.6);
	  \draw[thick] (0.0,0.0) -- (0.4,0.4); 
	  \draw[red, thick] (0.4,0.4) -- (0.6,0.6) -- (0.8,0.4) -- (1.0,0.6) -- (1.2,0.4);
	  \draw[thick] (1.2,0.4) -- (1.6,0.0) -- (1.8,0.2) -- (2.0,0.0);
	\end{scope}

	\begin{scope}[shift={(-3.4,-0.3)}]
	  \draw[step=0.2,black,ultra thin] (0.0,0.0) grid (2.0,0.6);
	  \draw[blue, thick] (0,0) -- (0.2,0.2) -- (0.4,0.0) -- (0.6,0.2);
	  \draw[thick] (0.6,0.2) -- (1.0,0.6) -- (1.4,0.2) -- (1.6,0.4) -- (1.8,0.2);
	  \draw[blue, thick] (1.8,0.2) -- (2.0,0.0);
	\end{scope}

	\node at (-0.7,1.0) {$d_\bot$};
	\node at (-0.7,0.7) {$\longmapsfrom$};
	\node at (-0.7,-0.2) {$d_\top$};
	\node at (-0.7,0.1) {$\longmapsfrom$};
  \end{tikzpicture}
\end{center}

  Then $X := j\lowershriek (A)$ has $X_1$ the set of all Dyck
  paths (all lengths and all heights), $X_2$ is the set of all Dyck paths with a
  marked level, and more generally $X_k$ is the set of all Dyck paths with $k-1$
  marked levels (which may coincide). The inner face maps delete a level marking
  (without affecting the path), whereas the outer face maps clip the path
  outside the outermost level. For example, the outer face maps involved in the
  formula for comultiplication in the incidence coalgebra, namely
  $\begin{tikzcd}[cramped]X_1 & X_2 \ar[l, shift left] \ar[l, shift
  right]\end{tikzcd}$ , are exemplified here:
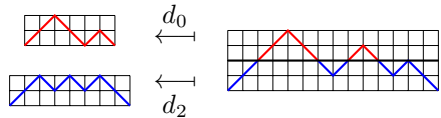
\begin{center}
  \begin{tikzpicture}
	\begin{scope}
	  \draw[step=0.2,black,ultra thin] (0.0,0.0) grid (2.8,0.8);
	  \draw[blue, thick] (0,0) -- (0.2,0.2) -- (0.4,0.4);
	  \draw[red, thick] (0.4,0.4) -- (0.6,0.6) -- (0.8,0.8) -- (1.0,0.6) -- (1.2,0.4);
	  \draw[blue, thick] (1.2,0.4) -- (1.4,0.2) -- (1.6,0.4);
	  \draw[red, thick] (1.6,0.4) -- (1.8,0.6) -- (2.0,0.4);
	  \draw[blue, thick] (2.0,0.4) -- (2.2,0.2) -- (2.4,0.4) -- (2.6,0.2) -- (2.8,0.0);
	  \draw[thick] (0.0,0.4) -- (2.8,0.4);
	\end{scope}
	
	\begin{scope}[shift={(-2.7,0.6)}]
	  \draw[step=0.2,black,ultra thin] (0.0,0.0) grid (1.2,0.4);
	  \draw[red, thick] (0.0,0.0) 
	  -- (0.2,0.2) -- (0.4,0.4) -- (0.6,0.2) -- (0.8,0.0) -- (1.0,0.2) -- (1.2,0.0);
	\end{scope}

	\begin{scope}[shift={(-2.9,-0.2)}]
	  \draw[step=0.2,black,ultra thin] (0.0,0.0) grid (1.6,0.4);
	  \draw[blue, thick] (0,0) -- (0.2,0.2) -- (0.4,0.4) -- (0.6,0.2) -- (0.8,0.4)
	  -- (1.0,0.2) -- (1.2,0.4) -- (1.4,0.2) -- (1.6,0.0);
	\end{scope}

	\node at (-0.7,1.0) {$d_0$};
	\node at (-0.7,0.7) {$\longmapsfrom$};
	\node at (-0.7,-0.2) {$d_2$};
	\node at (-0.7,0.1) {$\longmapsfrom$};
  \end{tikzpicture}
\end{center}

  There is another way to assemble Dyck paths into a free decomposition space: A
  {\em baseline point} of a Dyck path is one with second coordinate $0$. Let
  $A_n$ be the set of Dyck paths (any length and height) with $n-1$ baseline
  points (and $A_0$ consists of the trivial Dyck path). Then $X:=j\lowershriek
  (A)$ has in degree $1$ the set of all Dyck paths, and in degree $2$ the set of
  all Dyck paths with a chosen baseline point. The inner face maps forget
  baseline points, and the outer face maps delete the portion before the first
  or after the last chosen baseline point. For example
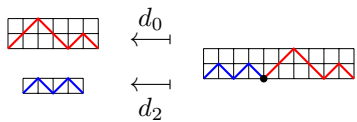
\begin{center}
  \begin{tikzpicture}
	\begin{scope}[shift={(0.0,0.2)}]
	  \draw[step=0.2,black,ultra thin] (0.0,0.0) grid (2.0,0.4);
	  \draw[blue, thick] (0,0) -- (0.2,0.2) -- (0.4,0.0) -- (0.6,0.2)-- (0.8,0.0);
	  \draw[red, thick] (0.8,0.0) -- (1.0,0.2) -- (1.2,0.4) -- (1.4,0.2) 
	  -- (1.6,0.0) -- (1.8,0.2) -- (2.0,0.0);
	  \fill (0.8,0.0) circle[radius=0.05];
	\end{scope}
	
	\begin{scope}[shift={(-2.6,0.6)}]
	  \draw[step=0.2,black,ultra thin] (0.0,0.0) grid (1.2,0.4);
	  \draw[red, thick] (0.0,0.0) 
	  -- (0.2,0.2) -- (0.4,0.4) -- (0.6,0.2) -- (0.8,0.0) -- (1.0,0.2) -- (1.2,0.0);
	\end{scope}

	\begin{scope}[shift={(-2.4,0.005)}]
	  \draw[step=0.2,black,ultra thin] (0.0,0.0) grid (0.8,0.2);
	  \draw[blue, thick] (0,0) -- (0.2,0.2) -- (0.4,0.0) -- (0.6,0.2)-- (0.8,0.0);
	\end{scope}

	\node at (-0.7,1.0) {$d_0$};
	\node at (-0.7,0.7) {$\longmapsfrom$};
	\node at (-0.7,-0.2) {$d_2$};
	\node at (-0.7,0.1) {$\longmapsfrom$};
  \end{tikzpicture}
\end{center}

  This construction is actually just an instance of the
  general word example, namely where the alphabet is the set of irreducible Dyck
  paths (meaning Dyck paths whose only baseline points are the start and the
  finish). The first construction is not of this form, as can be seen by the fact
  that an element in $X_2$ contains more information than its two layers.
\end{blanko}

\begin{blanko}{Layered posets (and quasi-symmetric functions, II).}\label{blanko layered posets}
  An {\em $\mathbf n$-layered poset} is a poset $P$ equipped with a monotone map
  to an ordinal $\mathbf n$; the fibers of this map are referred to as layers.
  Let $A_n$ denote the set of (iso-classes of) $\mathbf n$-layered posets.
  We define face maps
  $\begin{tikzcd}[cramped]A_{n-1} & A_n \ar[l, shift left] \ar[l,
  shift right]\end{tikzcd}$
  by
  deleting all elements in layer $n$, respectively deleting all elements in 
  layer $1$ (and
  then shifting down all layers to obtain an $(\mathbf{n-1}$)-layered poset. In the
  resulting free decomposition space $X:=j\lowershriek (A)$, we have $X_1$ the
  set of all layered posets (for all $n\in \N$) and $X_2$ the set of layered
  posets with an ordinal-sum splitting of the ordinal $\mathbf n$ into a initial segment 
  and a final segment.
  The resulting comultiplication is given by
  $$
  \Delta(P) = \sum_{\mathbf{a}+\mathbf{b}=\mathbf{n}} P_{\mid \mathbf a} 
  \otimes P_{\mid \mathbf b} ,
  $$
  where the sum is over ordinal-sum splittings,  
  and $P_{\mid \mathbf a}$ and $ P_{\mid \mathbf b}$ denote
  the restriction of $P$ along the ordinal-sum inclusions $\mathbf a \subset 
  \mathbf n \supset \mathbf b$.

  Many variations on this construction are possible, for example by demanding
  the monotone maps $P \to \mathbf n$ to be injective, surjective, bijective, or
  by putting constraints on $P$. If we require $P$ to be a linear poset and
  demand $P \to \mathbf n$ to be surjective, then we recover again the
  decomposition space $Q$ of quasi-symmetric functions from
  Example~\ref{QSym:I}. Indeed, to give a surjective monotone map $\mathbf m \to
  \mathbf n$ is the same as giving the length-$n$ word $(m_1,\ldots,m_n)$ of the
  sizes of the layers, so that the set $A_n = \{ \mathbf m \to \mathbf n\}$ is
  identified with $W(\N_+)_n$, and we end up with $j\lowershriek (A) = Q$. In
  fact, the viewpoint of monotone surjections is the original description of $Q$
  from \cite{GKT:QSym}.
\end{blanko}

\begin{blanko}{Heap orders, scheduling, and sequential processes.}
  In the previous class of examples,
  if we allow general posets $P$, but demand $P\to \mathbf n$ to be bijective,
  we arrive precisely at the notion of {\em heap-order}, i.e.~a order-compatible 
  numbering of the elements in the poset. (See also Stanley~\cite{Stanley:1972} 
  who used the terminology `labelled posets'.)
  In computer science, this is also called a
  `topological sort'~\cite{Knuth:TAOCP1}. They come up
  in connection with sorting and efficient data structures
  (most notably heap-ordered planar binary trees). A similar idea occurs in
  scheduling problems. In this case, each poset element is interpreted as a 
  task, and the order relation of the poset expresses dependency among tasks.
  A map $P \to \mathbf n$ then encodes a scheduling of the tasks subject to 
  the constraints.
  
  Mild generalizations cover many other situations where one
  considers sequences of computation steps, such as labelled
  transition systems (see for example~\cite{Winskel-Nielsen:1995} or
  \cite{HandbookOfProcessAlgebra}), Petri nets (see for
  example~\cite{DBLP:books/sp/Reisig85a}), or rewrite systems (such as double-pushout rewrite systems of graphs or other adhesive categories,
  as in~\cite{DBLP:conf/gg/CorradiniMREHL97}). It would take us too far afield
  to go into details with these examples, but the idea is clear: in each case
  there is a $\Deltainert\op$-diagram $A$, where $A_n$ is the set of sequences of
  length $n$, and where the face maps shorten the sequence at its beginning
  or at its end. In some situations
  one cannot simply concatenate such sequences, but deconcatenation is always possible,
  and the decomposition-space viewpoint can then be fruitful
  (see \cite{Kock:2005.05108} for Petri nets, and \cite{BehrKock}
  for double-pushout rewriting, where freeness of the decomposition spaces
  actually plays a role).
\end{blanko}

\subsection{Simplices in simplicial sets}

\begin{blanko}{Decomposition space of simplices.}
  For any simplicial set $X$, one can 
  restrict to $\Deltainert \subset \simplexcategory$ and then left Kan extend back to get the free
  decomposition space $j\lowershriek j\upperstar (X)$, the decomposition space of
  all simplices of $X$. If $X$ is the nerve of a category, this construction
  gives the free category on the underlying directed graph of $X$. The following
  subtle variation is much richer, though.
\end{blanko}

\begin{blanko}{The decomposition space of nondegenerate simplices.}
  Instead of considering all simplices, we restrict to nondegenerate simplices.
  Suppose that $X$ is a simplicial set with the property that outer faces of nondegenerate simplices are again nondegenerate.\footnote{Or more generally, $X$ could be a  `complete stiff simplicial space' \cite[\S4]{GKT2}.}
  This occurs, for instance, when $X$ is a discrete decomposition space.
  Let $\nondeg X_n \subset X_n$ denote the set of nondegenerate simplices.
  One can form the {\em
  decomposition space $J(X)$ of nondegenerate simplices}  
  \cite{GKT:QSym}, by first constructing the
  $\Deltainert\op$-diagram
  $$
  \begin{tikzcd}[cramped]\nondeg X_0 & \nondeg X_1 \ar[l, shift
  left] \ar[l, shift right]& \nondeg X_2 \ar[l, shift
  left] \ar[l, shift right] & \cdots\end{tikzcd}
  $$
  and then setting $J(X) := j\lowershriek(\nondeg X)$. 
  We have
  \begin{eqnarray*}
  J(X)_0 & = & X_0 \\
  J(X)_1 & = & \sum_{n\in \N} \nondeg X_n ,
  \end{eqnarray*}
  the set of all nondegenerate simplices in all dimensions.
  In higher simplicial dimension, it has {\em subdivided} nondegenerate 
  simplices. Precisely, a $k$-simplex of $J(X)$ is a diagram
  $$
  \Delta^k \actto \Delta^n \stackrel{\text{nondeg}}{\longrightarrow} X
  $$
  where the second map determines the nondegenerate $n$-simplex in $X$ and the 
  active map $\Delta^k \actto \Delta^n$ gives the subdivision, whereby 
  the $n$-simplex is subdivided into $k$ `stages.' (Note that if $X$ is Segal, 
  then $J(X)$ will be Segal again.)

  The decomposition space of nondegenerate simplices seems to be interesting in 
  general. (See \cite{BehrKock} for a recent use in rewriting theory.)
  One reason for its importance is the general link with 
  quasi-symmetric functions:
\end{blanko}

\begin{blanko}{Quasi-symmetric functions, III.}
  In the special case where the decomposition space $X$ is the nerve of the monoid
  $(\N,+)$, then the $\Deltainert\op$-diagram of nondegenerate simplices
  has $\N_+^n$ in degree $n$, and the free decomposition space is
  $$
  J(B\N) = Q
  $$ 
  the decomposition space of words in $\N_+$,
  whose
  incidence coalgebra is $\QSym$, the coalgebra of quasi-symmetric
  functions already visited in \ref{QSym:I} and \ref{blanko layered posets} (cf.~\cite{GKT:QSym}).
\end{blanko}

\bigskip

\noindent {\bf Acknowledgments.} 
This work was supported by a grant from the Simons Foundation (\#850849, PH).
JK gratefully acknowledges support from grants
MTM2016-80439-P and PID2020-116481GB-I00 (AEI/FEDER, UE) of Spain and
2017-SGR-1725 of Catalonia, and was also supported through the Severo Ochoa and
Mar\'ia de Maeztu Program for Centers and Units of Excellence in R\&D grant
number CEX2020-001084-M.

\bibliographystyle{scplain}
\bibliography{2-segal}

\begin{thebibliography}{10}

\bibitem{Aguiar-Mahajan}
{\sc Marcelo Aguiar {\rm and }Swapneel Mahajan}.
\newblock {\em Monoidal functors, species and {H}opf algebras}, vol.~29 of CRM
  Monograph Series.
\newblock American Mathematical Society, Providence, RI, 2010.
\newblock With forewords by Kenneth Brown and Stephen Chase and Andr{\'e}
  Joyal.
\newblock doi:10.1090/crmm/029.

\bibitem{AyalaFrancis:Fibrations}
{\sc David Ayala {\rm and }John Francis}.
\newblock {\em Fibrations of {$\infty$}-categories}.
\newblock High. Struct. {\bf 4} (2020), 168--265.

\bibitem{BehrKock}
{\sc Nicolas Behr {\rm and }Joachim Kock}.
\newblock {\em Tracelet {H}opf algebras and decomposition spaces}.
\newblock In {\em Proceedings of the Fourth International Conference on Applied
  Category Theory ACT2021 (Cambridge, 2021)}, vol. 372 of Electr. Proc.
  Theoret. Comput. Sci., pp. 323--337, 2022.
\newblock doi:10.4204/EPTCS.372.23, arXiv:2105.06186.

\bibitem{Berger:Cellular}
{\sc Clemens Berger}.
\newblock {\em A cellular nerve for higher categories}.
\newblock Adv. Math. {\bf 169} (2002), 118--175.
\newblock doi:10.1006/aima.2001.2056.

\bibitem{Berger-Mellies-Weber:1101.3064}
{\sc Clemens Berger, Paul-Andr{\'e} Melli{\`e}s, {\rm and }Mark Weber}.
\newblock {\em Monads with arities and their associated theories}.
\newblock J. Pure Appl. Algebra {\bf 216} (2012), 2029--2048.
\newblock doi:10.1016/j.jpaa.2012.02.039, arXiv:1101.3064.

\bibitem{Bergeron-Labelle-Leroux}
{\sc Fran{\c{c}}ois Bergeron, Gilbert Labelle, {\rm and }Pierre Leroux}.
\newblock {\em Combinatorial species and tree-like structures}, vol.~67 of
  Encyclopedia of Mathematics and its Applications.
\newblock Cambridge University Press, Cambridge, 1998.
\newblock Translated from the 1994 French original by Margaret Readdy, With a
  foreword by Gian-Carlo Rota.
\newblock doi:10.1017/CBO9781107325913.

\bibitem{Bergeron-Zabrocki}
{\sc Nantel Bergeron {\rm and }Mike Zabrocki}.
\newblock {\em The {H}opf algebras of symmetric functions and quasi-symmetric
  functions in non-commutative variables are free and co-free}.
\newblock J. Algebra Appl. {\bf 8} (2009), 581--600.
\newblock doi:10.1142/S0219498809003485, arXiv:math/0509265.

\bibitem{Bergner-Osorno-Ozornova-Rovelli-Scheimbauer:1609.02853}
{\sc Julia~E. Bergner, Ang\'{e}lica~M. Osorno, Viktoriya Ozornova, Martina
  Rovelli, {\rm and }Claudia~I. Scheimbauer}.
\newblock {\em 2-{S}egal sets and the {W}aldhausen construction}.
\newblock Topology Appl. {\bf 235} (2018), 445--484.
\newblock doi:10.1016/j.topol.2017.12.009, arXiv:1609.02853.

\bibitem{BOORS:Edgewise}
{\sc Julia~E. Bergner, Ang\'{e}lica~M. Osorno, Viktoriya Ozornova, Martina
  Rovelli, {\rm and }Claudia~I. Scheimbauer}.
\newblock {\em The edgewise subdivision criterion for 2-{S}egal objects}.
\newblock Proc. Amer. Math. Soc. {\bf 148} (2020), 71--82.
\newblock doi:10.1090/proc/14679, arXiv:1807.05069.

\bibitem{HandbookOfProcessAlgebra}
J.~A. Bergstra, A.~Ponse, {\rm and }S.~A. Smolka, editors.
\newblock {\em Handbook of process algebra}.
\newblock North-Holland Publishing Co., Amsterdam, 2001.
\newblock doi:10.1016/B978-0-444-82830-9.X5017-6.

\bibitem{Bunge-Fiore}
{\sc Marta Bunge {\rm and }Marcelo Fiore}.
\newblock {\em Unique factorisation lifting functors and categories of
  linearly-controlled processes}.
\newblock Math. Struct. Comput. Sci. {\bf 10} (2000), 137--163.
\newblock doi:10.1017/S0960129599003023.

\bibitem{Bunge-Niefield}
{\sc Marta Bunge {\rm and }Susan Niefield}.
\newblock {\em Exponentiability and single universes}.
\newblock J. Pure Appl. Algebra {\bf 148} (2000), 217–250.
\newblock doi:10.1016/S0022-4049(98)00172-8.

\bibitem{Burkin:Twisted}
{\sc Sergei Burkin}.
\newblock {\em Twisted arrow categories, operads and {S}egal conditions}.
\newblock Theory Appl. Categ. {\bf 38} (2022), Paper No. 16, 595--660.

\bibitem{Chu-Haugseng:1907.03977}
{\sc Hongyi Chu {\rm and }Rune Haugseng}.
\newblock {\em Homotopy-coherent algebra via {S}egal conditions}.
\newblock Adv. Math. {\bf 385} (2021), 107733.
\newblock doi:10.1016/j.aim.2021.107733, arXiv:1907.03977.

\bibitem{DBLP:conf/gg/CorradiniMREHL97}
{\sc Andrea Corradini, Ugo Montanari, Francesca Rossi, Hartmut Ehrig, Reiko
  Heckel, {\rm and }Michael L{\"{o}}we}.
\newblock {\em Algebraic approaches to graph transformation - Part {I:} basic
  concepts and double pushout approach}.
\newblock In Grzegorz Rozenberg, editor, {\em Handbook of Graph Grammars and
  Computing by Graph Transformations, Volume 1: Foundations}, pp. 163--246.
  World Scientific, 1997.
\newblock doi:10.1142/9789812384720\_0003.

\bibitem{Dyckerhoff-Kapranov:1212.3563}
{\sc Tobias Dyckerhoff {\rm and }Mikhail Kapranov}.
\newblock {\em Higher {S}egal spaces}, vol. 2244 of Lecture Notes in
  Mathematics.
\newblock Springer, Cham, 2019.
\newblock doi:10.1007/978-3-030-27124-4, arXiv:1212.3563.

\bibitem{FGKPW}
{\sc Matthew Feller, Richard Garner, Joachim Kock, May~U. Proulx, {\rm and
  }Mark Weber}.
\newblock {\em Every 2-{S}egal space is unital}.
\newblock Commun. Contemp. Math. {\bf 23} (2021), 2050055, 6.
\newblock doi:10.1142/S0219199720500558, arXiv:1905.09580.

\bibitem{GKT:QSym}
{\sc Imma G\'{a}lvez-Carrillo, Joachim Kock, {\rm and }Andrew Tonks}.
\newblock Decomposition spaces of quasi-symmetric functions.
\newblock Unpublished/in preparation.

\bibitem{GKT:ex}
{\sc Imma G{\'a}lvez-Carrillo, Joachim Kock, {\rm and }Andrew Tonks}.
\newblock {\em Decomposition spaces in combinatorics}.
\newblock Preprint, arXiv:1612.09225.

\bibitem{GKT1}
{\sc Imma G\'{a}lvez-Carrillo, Joachim Kock, {\rm and }Andrew Tonks}.
\newblock {\em Decomposition spaces, incidence algebras and {M}\"{o}bius
  inversion {I}: {B}asic theory}.
\newblock Adv. Math. {\bf 331} (2018), 952--1015.
\newblock doi:10.1016/j.aim.2018.03.016, arXiv:1512.07573.

\bibitem{GKT2}
{\sc Imma G{\'a}lvez-Carrillo, Joachim Kock, {\rm and }Andrew Tonks}.
\newblock {\em Decomposition spaces, incidence algebras and {M}\"{o}bius
  inversion {II}: {C}ompleteness, length filtration, and finiteness}.
\newblock Adv. Math. {\bf 333} (2018), 1242--1292.
\newblock doi:10.1016/j.aim.2018.03.017, arXiv:1512.07577.

\bibitem{GKT3}
{\sc Imma G\'{a}lvez-Carrillo, Joachim Kock, {\rm and }Andrew Tonks}.
\newblock {\em Decomposition spaces, incidence algebras and {M}\"{o}bius
  inversion {III}: {T}he decomposition space of {M}\"{o}bius intervals}.
\newblock Adv. Math. {\bf 334} (2018), 544--584.
\newblock doi:10.1016/j.aim.2018.03.018, arXiv:1512.07580.

\bibitem{GKT:restr}
{\sc Imma G\'{a}lvez-Carrillo, Joachim Kock, {\rm and }Andrew Tonks}.
\newblock {\em Decomposition spaces and restriction species}.
\newblock Int. Math. Res. Notices {\bf 2020} (2020), 7558--7616.
\newblock doi:10.1093/imrn/rny089, arXiv:1708.02570.

\bibitem{Gepner-Haugseng-Kock:1712.06469}
{\sc David Gepner, Rune Haugseng, {\rm and }Joachim Kock}.
\newblock {\em $\infty$-Operads as analytic monads}.
\newblock Int. Math. Res. Notices {\bf 2022} (2022), 12516--12624.
\newblock doi:10.1093/imrn/rnaa332, arXiv:1712.06469.

\bibitem{Hackney:2208.13852}
{\sc Philip Hackney}.
\newblock {\em Segal conditions for generalized operads}.
\newblock Preprint, arXiv:2208.13852, to appear in Higher Structures in
  Geometry, Topology and Physics, Contemp. Math. AMS.

\bibitem{HK-untwist}
{\sc Philip Hackney {\rm and }Joachim Kock}.
\newblock {\em Culf maps and edgewise subdivision}.
\newblock Preprint, arXiv:2210.11191, with an appendix coauthored with Jan
  Steinebrunner.

\bibitem{Hackney-Robertson-Yau}
{\sc Philip Hackney, Marcy Robertson, {\rm and }Donald Yau}.
\newblock {\em Infinity properads and infinity wheeled properads}, vol. 2147 of
  Lecture Notes in Mathematics.
\newblock Springer, Cham, 2015.
\newblock doi:10.1007/978-3-319-20547-2, arXiv:1410.6716.

\bibitem{HACKNEY2020107206}
{\sc Philip Hackney, Marcy Robertson, {\rm and }Donald Yau}.
\newblock {\em Modular operads and the nerve theorem}.
\newblock Adv. Math. {\bf 370} (2020), 107206.
\newblock doi:10.1016/j.aim.2020.107206, arXiv:1906.01144.

\bibitem{Hivert-Novelli-Thibon:0605262}
{\sc Florent Hivert, Jean-Christophe Novelli, {\rm and }Jean-Yves Thibon}.
\newblock {\em Commutative combinatorial {H}opf algebras}.
\newblock J. Algebraic Combin. {\bf 28} (2008), 65--95.
\newblock doi:10.1007/s10801-007-0077-0, arXiv:0605262.

\bibitem{Hoang:2005.01198}
{\sc Truong Hoang}.
\newblock {\em Quillen cohomology of enriched operads}.
\newblock Preprint, arXiv:2005.01198.

\bibitem{Johnstone:Conduche'}
{\sc Peter Johnstone}.
\newblock {\em A note on discrete {C}onduch\'e fibrations}.
\newblock Theory Appl. Categ. {\bf 5} (1999), No.\ 1, 1--11.

\bibitem{Joni-Rota}
{\sc Saj-nicole~A. Joni {\rm and }Gian-Carlo Rota}.
\newblock {\em Coalgebras and bialgebras in combinatorics}.
\newblock Stud. Appl. Math. {\bf 61} (1979), 93--139.
\newblock doi:10.1002/sapm197961293.

\bibitem{Joyal:1981}
{\sc Andr{\'e} Joyal}.
\newblock {\em Une th\'eorie combinatoire des s\'eries formelles}.
\newblock Adv. Math. {\bf 42} (1981), 1--82.
\newblock doi:10.1016/0001-8708(81)90052-9.

\bibitem{quadern45}
{\sc Andr{\'e} Joyal}.
\newblock {\em The theory of quasi-categories and its applications}.
\newblock No.~45 in Quaderns. CRM, Barcelona, 2008.
\newblock Available at
  \url{http://mat.uab.cat/~kock/crm/hocat/advanced-course/Quadern45-2.pdf}.

\bibitem{Knuth:TAOCP1}
{\sc Donald~E. Knuth}.
\newblock {\em The Art of Computer Programming, Vol. 1: Fundamental
  Algorithms}.
\newblock Addison-Wesley, Reading, Mass., third edition, 1997.

\bibitem{Kock:0807}
{\sc Joachim Kock}.
\newblock {\em Polynomial functors and trees}.
\newblock Int. Math. Res. Notices {\bf 2011} (2011), 609--673.
\newblock doi:10.1093/imrn/rnq068, arXiv:0807.2874.

\bibitem{Kock:1407.3744}
{\sc Joachim Kock}.
\newblock {\em Graphs, hypergraphs, and properads}.
\newblock Collect. Math. {\bf 67} (2016), 155--190.
\newblock doi:10.1007/s13348-015-0160-0, arXiv:1407.3744.

\bibitem{Kock:1912.11320}
{\sc Joachim Kock}.
\newblock {\em The incidence comodule bialgebra of the {B}aez--{D}olan
  construction}.
\newblock Adv. Math. {\bf 383} (2021), Paper No. 107693, 79.
\newblock doi:10.1016/j.aim.2021.107693, arXiv:1912.11320.

\bibitem{Kock:2005.05108}
{\sc Joachim Kock}.
\newblock {\em Whole-grain Petri nets and processes}.
\newblock J. ACM. {\bf 70} (2022), 1--58.
\newblock doi:10.1145/3559103, arXiv:2005.05108.

\bibitem{Kock-Spivak:1807.06000}
{\sc Joachim Kock {\rm and }David~I. Spivak}.
\newblock {\em Decomposition-space slices are toposes}.
\newblock Proc. Amer. Math. Soc. {\bf 148} (2020), 2317--2329.
\newblock doi:10.1090/proc/14834, arXiv:1807.06000.

\bibitem{Lawvere:statecats}
{\sc F.~William Lawvere}.
\newblock {\em State categories and response functors. Dedicated to Walter
  Noll.}
\newblock Preprint (May 1986).

\bibitem{Lawvere-Menni}
{\sc F.~William Lawvere {\rm and }Mat{\'\i}as Menni}.
\newblock {\em The {H}opf algebra of {M}{\"o}bius intervals}.
\newblock Theory Appl. Categ. {\bf 24} (2010), 221--265.

\bibitem{Leroux:1976}
{\sc Pierre Leroux}.
\newblock {\em Les cat{\'e}gories de {M}{\"o}bius}.
\newblock Cahiers Topol. G{\'e}om. Diff. {\bf 16} (1976), 280--282.

\bibitem{HTT}
{\sc Jacob Lurie}.
\newblock {\em Higher topos theory}, vol. 170 of Annals of Mathematics Studies.
\newblock Princeton University Press, Princeton, NJ, 2009.
\newblock doi:10.1515/9781400830558.

\bibitem{Lurie:HA}
{\sc Jacob Lurie}.
\newblock {\em Higher algebra}.
\newblock Available from \url{http://www.math.harvard.edu/~lurie/}, 2013.

\bibitem{Malvenuto-Reutenauer}
{\sc Claudia Malvenuto {\rm and }Christophe Reutenauer}.
\newblock {\em Duality between quasi-symmetric functions and the {S}olomon
  descent algebra}.
\newblock J. Algebra {\bf 177} (1995), 967--982.
\newblock doi:10.1006/jabr.1995.1336.

\bibitem{Novelli-Thibon:PQSym}
{\sc Jean-Christophe Novelli {\rm and }Jean-Yves Thibon}.
\newblock {\em Hopf algebras and dendriform structures arising from parking
  functions}.
\newblock Fund. Math. {\bf 193} (2007), 189--241.
\newblock doi:10.4064/fm193-3-1, arXiv:0511200.

\bibitem{DBLP:books/sp/Reisig85a}
{\sc Wolfgang Reisig}.
\newblock {\em Petri Nets: An Introduction}, vol.~4 of {EATCS} Monographs on
  Theoretical Computer Science.
\newblock Springer, 1985.
\newblock doi:10.1007/978-3-642-69968-9.

\bibitem{Rota:Moebius}
{\sc Gian-Carlo Rota}.
\newblock {\em On the foundations of combinatorial theory. {I}. {T}heory of
  {M}\"obius functions}.
\newblock Z. Wahrscheinlichkeitstheorie und Verw. Gebiete {\bf 2} (1964),
  340--368.
\newblock doi:10.1007/BF00531932.

\bibitem{Schmitt:hacs}
{\sc William~R. Schmitt}.
\newblock {\em Hopf algebras of combinatorial structures}.
\newblock Canad. J. Math. {\bf 45} (1993), 412--428.
\newblock doi:10.4153/CJM-1993-021-5.

\bibitem{Stanley:1972}
{\sc Richard Stanley}.
\newblock {\em Ordered structures and partitions}.
\newblock Memoirs of the American Mathematical Society, no. 119. American
  Mathematical Society, Providence, 1972.

\bibitem{Stanley:PFNP}
{\sc Richard~P. Stanley}.
\newblock {\em Parking functions and noncrossing partitions}.
\newblock Electron. J. Combin. {\bf 4} (1997), Research Paper 20, 1--14.
\newblock The Wilf Festschrift (Philadelphia, PA, 1996).
\newblock doi:10.37236/1335.

\bibitem{Stanley:volII}
{\sc Richard~P. Stanley}.
\newblock {\em Enumerative combinatorics. {V}ol. 2}, vol.~62 of Cambridge
  Studies in Advanced Mathematics.
\newblock Cambridge University Press, Cambridge, 1999.
\newblock doi:10.1017/CBO9780511609589.

\bibitem{Street:categorical-structures}
{\sc Ross Street}.
\newblock {\em Categorical structures}.
\newblock In {\em Handbook of algebra, {V}ol.\ 1}, pp. 529--577. North-Holland,
  Amsterdam, 1996.
\newblock doi:10.1016/S1570-7954(96)80019-2.

\bibitem{Thomason:notebook85}
{\sc Robert Thomason}.
\newblock Notebook 85, 1995.
\newblock
  \url{https://www.math-info-paris.cnrs.fr/bibli/digitization-of-robert-wayne-thomasons-notebooks/}.

\bibitem{Weber:TAC13}
{\sc Mark Weber}.
\newblock {\em Generic morphisms, parametric representations and weakly
  {C}artesian monads}.
\newblock Theory Appl. Categ. {\bf 13} (2004), 191--234.

\bibitem{Weber:TAC18}
{\sc Mark Weber}.
\newblock {\em Familial 2-functors and parametric right adjoints}.
\newblock Theory Appl. Categ. {\bf 18} (2007), 665--732.

\bibitem{Winskel-Nielsen:1995}
{\sc Glynn Winskel {\rm and }Mogens Nielsen}.
\newblock {\em Models for concurrency}.
\newblock In {\em Handbook of logic in computer science}, vol.~4, pp. 1--148.
  Oxford Univ. Press, New York, 1995.

\end{thebibliography}

\end{document}